\documentclass{amsart}%
\usepackage{amsfonts}
\usepackage{amsmath}
\usepackage{amssymb}
\usepackage{mathtools}
\usepackage[usenames,dvipsnames]{xcolor}

\usepackage{graphicx}
\usepackage{hyperref}
\usepackage{microtype}
\usepackage{float}

\usepackage{tikz}							

\hypersetup{
    unicode=false,          		
    linktocpage=true,			
    pdfnewwindow=true,      	
    colorlinks=true,       			
    linkcolor=red,          		
    citecolor=PineGreen,    	
    filecolor=magenta,      		
    urlcolor=cyan           		
}

\newtheorem{theorem}{Theorem}[section]
\theoremstyle{plain}
\newtheorem{corollary}[theorem]{Corollary}
\newtheorem{definition}[theorem]{Definition}
\newtheorem{lemma}[theorem]{Lemma}
\newtheorem{proposition}[theorem]{Proposition}

\theoremstyle{remark}
\newtheorem{remark}[theorem]{Remark}

\numberwithin{equation}{section}

\DeclareMathOperator{\Tr}{Tr}
\DeclareMathOperator{\Det}{Det}

\DeclareMathOperator{\dist}{dist}

\let\emptyset\varnothing

\begin{document}
\title[Solitons for the Focusing DS II Equation]{Soliton Solutions and their (In)stability for the Focusing Davey-Stewartson II Equation}
\dedicatory{With an Appendix by Russell Brown}
\author{Peter A. Perry}
\address[Perry]{Department of Mathematics, University of Kentucky, Lexington, Kentucky 40506--0027}
\address[Brown]{Department of Mathematics, University of Kentucky, Lexington, Kentucky 40506--0027}
\thanks{Perry supported in part by NSF\ grant  DMS-1208778.}
\thanks{Version of \today}

\begin{abstract}
We give a rigorous mathematical analysis of  the one-soliton solution of the focusing Davey-Stewartson II equation and a proof of its instability under perturbation. Building on the fundamental perturbation analysis of Gadyl'shin and Kiselev, we verify their Assumption 1 and use Fredholm determinants to globalize their perturbation analysis. 
\end{abstract}

\maketitle
\tableofcontents

%
%
%
%


\newcommand{\bbC}{\mathbb{C}}
\newcommand{\bbR}{\mathbb{R}}


\newcommand{\calA}{\mathcal{A}}
\newcommand{\calB}{\mathcal{B}}
\newcommand{\calC}{\mathcal{C}}
\newcommand{\calD}{\mathcal{D}}
\newcommand{\calE}{\mathcal{E}}
\newcommand{\calF}{\mathcal{F}}
\newcommand{\calG}{\mathcal{G}}
\newcommand{\calH}{\mathcal{H}}
\newcommand{\calI}{\mathcal{I}}
\newcommand{\calJ}{\mathcal{J}}
\newcommand{\calK}{\mathcal{K}}
\newcommand{\calL}{\mathcal{L}}
\newcommand{\calM}{\mathcal{M}}
\newcommand{\calN}{\mathcal{N}}
\newcommand{\calO}{\mathcal{O}}
\newcommand{\calP}{\mathcal{P}}
\newcommand{\calQ}{\mathcal{Q}}
\newcommand{\calR}{\mathcal{R}}
\newcommand{\calS}{\mathcal{S}}
\newcommand{\calT}{\mathcal{T}}
\newcommand{\calU}{\mathcal{U}}
\newcommand{\calV}{\mathcal{V}}
\newcommand{\calW}{\mathcal{W}}
\newcommand{\calX}{\mathcal{X}}
\newcommand{\calY}{\mathcal{Y}}
\newcommand{\calZ}{\mathcal{Z}}


\newcommand{\bfe}{\mathbf{e}}
\newcommand{\bfm}{\mathbf{m}}
\newcommand{\bft}{\mathbf{t}}


\newcommand{\bfA}{\mathbf{A}}
\newcommand{\bfB}{\mathbf{B}}
\newcommand{\bfC}{\mathbf{C}}
\newcommand{\bfD}{\mathbf{D}}
\newcommand{\bfE}{\mathbf{E}}
\newcommand{\bfF}{\mathbf{F}}
\newcommand{\bfG}{\mathbf{G}}
\newcommand{\bfH}{\mathbf{H}}
\newcommand{\bfI}{\mathbf{I}}
\newcommand{\bfJ}{\mathbf{J}}
\newcommand{\bfK}{\mathbf{K}}
\newcommand{\bfL}{\mathbf{L}}
\newcommand{\bfM}{\mathbf{M}}
\newcommand{\bfN}{\mathbf{N}}
\newcommand{\bfO}{\mathbf{O}}
\newcommand{\bfP}{\mathbf{P}}
\newcommand{\bfQ}{\mathbf{Q}}
\newcommand{\bfR}{\mathbf{R}}
\newcommand{\bfS}{\mathbf{S}}
\newcommand{\bfT}{\mathbf{T}}
\newcommand{\bfU}{\mathbf{U}}
\newcommand{\bfV}{\mathbf{V}}
\newcommand{\bfW}{\mathbf{W}}
\newcommand{\bfX}{\mathbf{X}}
\newcommand{\bfY}{\mathbf{Y}}
\newcommand{\bfZ}{\mathbf{Z}}


\newcommand{\abar}{\overline{a}}
\newcommand{\bbar}{\overline{b}}
\newcommand{\fbar}{{\overline{f}}}
\newcommand{\gbar}{{\overline{g}}}
\newcommand{\kbar}{{\overline{k}}}
\newcommand{\rbar}{{\overline{r}}}
\newcommand{\ubar}{{\overline{u}}}
\newcommand{\qbar}{{\overline{q}}}
\newcommand{\xbar}{\overline{x}}
\newcommand{\wbar}{\overline{w}}
\newcommand{\zbar}{{\overline{z}}}

\newcommand{\Pbar}{\overline{P}}


\newcommand{\mubar}{{\overline{\mu}}}
\newcommand{\xibar}{\overline{\xi}}
\newcommand{\nubar}{\overline{nu}}
\newcommand{\psibar}{\overline{\psi}}


\newcommand{\dee}{\partial}
\newcommand{\dbar}{\overline{\partial}}

\newcommand{\dint}{\displaystyle{\int}}
\newcommand{\dsum}{\displaystyle{\sum}}

\newcommand{\darr}{\downarrow}
\newcommand{\uarr}{\uparrow}
\newcommand{\rarr}{\rightarrow}

\newcommand{\eps}{\varepsilon}


\newcommand{\dotarg}{\, \cdot \, }
\newcommand{\diamondarg}{\, \diamond \, }


\newcommand{\Yeps}{{Y_\eps}}


\newcommand{\bigO}[1]{{\mathcal{O}}\left( {#1} \right)}
\newcommand{\norm}[1]{\left\Vert {#1} \right\Vert}
\newcommand{\absval}[1]{\left\vert {#1} \right\vert}


\newcommand{\twovec}[2]{\left( \begin{array}{c} {#1} \\ {#2} \end{array}\right)}
\newcommand{\Twovec}[2]{\left( \begin{array}{c} {#1}\\  \\ {#2} \end{array}\right)}


\newcommand{\twomat}[4]{\left( \begin{array}{rr} {#1} & {#2} \\ {#3} & {#4} \end{array} \right) }
\newcommand{\bigtwomat}[4]{\left( \begin{array}{cc} {#1} & {#2} \\ \\ {#3} & {#4} \end{array} \right) }
\newcommand{\diagmat}[2]{\left( \begin{array}{rr} {#1} & {0} \\ {0} & {#2} \end{array} \right) }


\newcommand{\sigmathree}{\twomat{1}{0}{0}{-1}}

\renewcommand{\norm}[2]{\| {#1} \|_{#2}}

\newcommand{\phibar}{\overline{\varphi}}
\newcommand{\alphabar}{\overline{\alpha}}
\newcommand{\betabar}{\overline{\beta}}
\newcommand{\chibar}{\overline{\chi}}
\newcommand{\kappabar}{\overline{\kappa}}
\newcommand{\zetabar}{\overline{\zeta}}

\newcommand{\mudot}{\dot{\mu}}

\newcommand{\Cauchy}{\calC}
\newcommand{\CauchyBar}{\overline{\calC}}

\newcommand{\Bpert}{\left( \delta B^* \right)}

\newcommand{\angles}[2]{ \left\langle {#1} , {#2} \right\rangle}

\newcommand{\compose}{\, \circ \, }

\newcommand{\loc}{\mathrm{loc}}

\newcommand{\lam}{\lambda}

\newcommand{\Cauchybar}{\overline{\Cauchy}}

\newcommand{\Twomat}[4]
{
	\left(
		\begin{array}{ccc}
		#1	&&	#2	\\[5pt]
		#3	&&	#4
		\end{array}
	\right)
}

\newcommand{\offdiagmat}[2]{\twomat{0}{#1}{#2}{0}}

\section{Introduction}

In this paper we will give a rigorous proof that the one-soliton solution for the focusing Davey-Stewartson II (fDSII) equation in two-dimensions is unstable under smooth, compactly supported perturbations of the initial data. Our proof uses  the inverse scattering method and sharp asymptotic analysis for a renormalized Fredholm determinant whose zeros signal the presence of soliton solutions. As we will explain, our approach builds on previous work of Gadyl'shin and Kiselev \cite{GK:1999}.

We will study the fDSII equation in the form
\begin{align}
\label{fDSII}
iu_t 		&=	\left( \dee^2 + \dbar^2 \right) u + \left(g + \gbar\right) u\\
\nonumber
\dbar g 	&=	-\frac{1}{2} \dee \left( |u|^2 \right)
\end{align}
Here $u=u(x,y,t)$, $\dee=(1/2)\left( \dee_x - i \dee_y \right)$, and $\dbar=(1/2)\left(\dee_x+i\dee_y\right)$. 
The fDSII equation is the shallow-water limit of the two-dimensional, dispersive nonlinear PDE derived by Benny-Roskes \cite{BR:1969} and Davey-Stewartson \cite{DS:1974} to describe the propagation of weakly nonlinear, monochromatic surface waves. The solution $u(x,y,t)$ gives the amplitude envelope of the wave. We refer the reader to the paper of Ghidaglia and Saut \cite{GS:1989} for a physical derivation of the Davey-Stewartson equation and its limiting case, the fDSII equation.

The fDSII equation is completely integrable in the following sense.  Let $u \in C^1(\bbC \times \bbR)$  
be given and let  $g \in C^1(\bbC \times \bbR)$ with $\dbar g = -(1/2) \dee (|u|^2)$. Let
$$ J = \twomat{1}{0}{0}{-1}, \quad Q = \offdiagmat{u}{-\ubar}, \quad B =  \twomat{ig}{-i\dee g}{-i\dbar \ubar}{i\gbar}$$
A function $u \in C^1(\bbC \times \bbR)$ is 
a solution of the fDSII equation if and only if the relation
$$ \dot{L} = [A,L] $$
holds as operators, where
\begin{align*}
L 	&=	-\dee_x - iJ\dee_y + Q\\
A	&=	B - Q \dee_y +iJ\dee_y^2
\end{align*}

The scattering data used to linearize the fDSII flow are determined by $2 \times 2$ matrix-valued solutions
$\Psi(z,k)$ of the problem 
\begin{subequations}
\label{CGO}
\begin{align}
\label{CGO.pde}
L\Psi	&=	0,\\
\label{CGO.bc}
 \lim_{|z| \rarr \infty} \Psi(z,k) \diagmat{e^{-ikz}}{e^{i\kbar \zbar}} &= \diagmat{1}{1} 
\end{align}
\end{subequations}
where we now take for $u$ a generic function $u \in L^2(\bbC)$. 
Setting 
$$ M(z,k) = \Psi(z,k) \diagmat{e^{-ikz}}{e^{i\kbar \zbar}}, $$ 
we obtain the spectral problem
\begin{subequations}
\label{CGO.m}
\begin{align}
\label{CGO.m.pde}
\Twomat{\dbar M_{11}}{\left(\dbar - i\kbar\right)M_{12}}{\left(\dee + ik\right)M_{21}}{\dee M_{22}}
&=
\frac{1}{2} Q M,\\
\label{CGO.m.bc}
\lim_{|z| \rarr \infty} M(z,k) &= \diagmat{1}{1}.
\end{align}
\end{subequations}

Fix $p > 2$. We say that $k \in \bbC$ is a \emph{regular point} for the problem \eqref{CGO.m} if there exists 
a unique matrix-valued solution $M(\dotarg,k)$ with 
$$M(\dotarg,k) -\twomat{1}{0}{0}{1} \in L^p(\bbC),$$ and an 
\emph{exceptional point} if there is a solution $M(\dotarg,k)$ of \eqref{CGO.m} with $M(z,k) \in L^p(\bbC)$.
We denote by $Z$ the set of all exceptional points, called the \emph{exceptional set}. We will show that the 
exceptional set is closed and bounded.

We will say that an initial datum $u_0$ for \eqref{fDSII} \emph{supports solitons} if $Z$ is a nonempty, discrete set. Under ``small data'' conditions it can be shown that $Z = \emptyset$. On the other hand, Arkadiev, Progrebov, and Polivanov \cite{APP:1989} derived an explicit family of initial data 
\begin{equation}
\label{APP}
u_0(z)  = \frac{2 \overline{\nu_0} e_{k_0}(z)}{|z+ \mu_0|^2 + |\nu_0|^2},
\end{equation}
where $k_0$, $\nu_0$, $\mu_0$ are complex parameters, which have an exceptional point at $k_0$,  and give rise to soliton solutions of \eqref{fDSII} with algebraic decay. Here and in what follows we set
\begin{equation}
\label{ek}
e_k(z)= \exp\left( i(kz+\kbar \zbar) \right). 
\end{equation}
Gadyl'shin and Kiselev showed that these solutions are unstable in the sense that, for a set of 
$C_0^\infty(\bbC)$ perturbations with finite codimension, the exceptional set $Z$ for $u_0 +\eps \varphi$ does not contain $k_0$. Their proof is perturbative in nature and relies on an unproven assumption about the spectrum of a Fredholm operator associated to the problem \eqref{CGO.m}. Here we give a global analysis and, along the way, 
prove the spectral assumption made by Gadyl'shin and Kiselev. The formulation of (ii) below is due to Gadyl'shin and Kiselev but we globalize their result through the use of a Fredholm determinant associated to the direct scattering problem. We will prove:

\begin{theorem}
\label{thm:main}
(i) The potential $u_0(z)$ has $Z = \{ k_0 \}$. (ii) Choose $\varphi \in C_0^\infty(\bbC)$ so that 
$$ \int_{\bbC} \left(\chi - |z|^2 \chibar\right) \left(1+|z|^2\right)^{-2} \, dm(z), 
$$
is nonzero, where $\chi  = e_{-k_0} \varphi$. Then, for all sufficiently small $\eps>0$, the potential $u_\eps = u_0 + \eps \varphi$ has empty exceptional set.
\end{theorem}

To prove Theorem \ref{thm:main}, we reduce the study of the exceptional set to a renormalized determinant of 
a Fredholm operator associated to problem \eqref{CGO.m} through a series of (standard) symmetry reductions. 
We make use of the generalized determinant defined by Gohberg, Goldberg, and Krein \cite{GGK:1997,GGK:2000}
to define the determinant and show that it solves a $\dbar$ problem determined by scattering data. Using this equation we are able to compute the determinant associated to the soliton solution to show that $Z = \{ k_0 \}$. We then study the behavior of the determinant under perturbations in order to show that 
$Z = \emptyset$ for perturbations obeying the hypotheses of Theorem \ref{thm:main}. 

The paper is organized as follows. In \S \ref{sec:prelim} we fix notation, recall useful estimates on the $\dbar$ problem, and summarize relevant results of perturbation theory. In \S \ref{sec:direct} we study the direct scattering problem \eqref{CGO.m}, define the scattering data, and define the Fredholm determinant. In \S \ref{sec:one-soliton},
we compute the determinant of the one-soliton solution and prove the spectral assumption of Gadyl'shin and Kiselev.
Finally, in \S \ref{sec:perturb}, we study perturbations of the determinant and prove Theorem \ref{thm:main}.

\subsection*{Acknowledgments}
It is a pleasure to thank Russell Brown, Ken McLaughlin, Peter Miller, Michael Music, Katharine Ott, and Brad Schwer for helpful discussions.

\section{Preliminaries}
\label{sec:prelim}

\subsection{Notation} If $X$ and $Y$ are Banach spaces with $X \cap Y$ dense in $X$ and $Y$, we norm $X \cap Y$ with the norm $\norm{f}{X \cap Y} = \norm{f}{X}+ \norm{f}{Y}$. We denote by $\calB(X,Y)$ the Banach space of bounded linear operators from $X$ to $Y$, and by $\calB(X)$ the Banach space $\calB(X,X)$. We define Fourier transforms adapted to the $\dee$ and $\dbar$ operators:
\begin{align}
\label{F.direct}
\left(\calF f\right)(k) 	&=	\frac{1}{\pi} \int_\bbC e_{-k}(z) f(z) \, dm(z), \\[5pt]
\label{F.inverse}
\left(\calF^{-1} g\right) (z)	&=	\frac{1}{\pi} \int_{\bbC} e_k(z) g(k) \, dm(z)
\end{align}
where $dm(\dotarg)$ denotes Lebesgue measure on $\bbC$. 

We write $f \lesssim g$ to indicate an upper bound up to absolute numerical constants, and
$f \lesssim_{\, p} g$ to indicate an upper bound up to positive constants depending on $p$.

\subsection{Cauchy and Beurling Transforms} For $q \in (1,2)$, denote by $\tilde{q}$ the Sobolev conjugate given by  $\tilde{q}^{-1} = q^{-1}-1/2$. The solid Cauchy transform
\begin{equation}
\label{C}
\Cauchy \left[ f \right] (z) = \frac{1}{\pi} \int \frac{1}{z-\zeta} f(\zeta) \, dm(\zeta)
\end{equation}
satisfies $\dbar \circ \Cauchy = \Cauchy \circ \dbar = I$ on $C_0^\infty(\bbC)$. The following standard estimates
extend $\calC$ to larger function spaces and quantify the regularity of $\calC\left[ f \right]$. 
\begin{lemma}
\label{lemma:P1}
Suppose that $1<p<2<q<\infty$. For any $f \in L^{2q/(q+2)}(\bbC)$, 
\begin{equation}
\label{P1}
\norm{\Cauchy f}{q} \lesssim_{\, q} \norm{f}{2q/(q+2)}.
\end{equation}
For any $f \in L^q(\bbC)$, 
\begin{equation}
\label{P2}
\absval{ (\Cauchy f)(z) - (\Cauchy f)(w)} \lesssim_{\, q} \norm{f}{q} |z-w|^{(q-2)/q}.
\end{equation}
Finally, for any $f \in L^p(\bbC) \cap L^q(\bbC)$, 
\begin{equation}
\label{P3}
\norm{\Cauchy f}{\infty} \lesssim_{\, p, q}  \norm{f}{L^p \cap L^q}.
\end{equation}
\end{lemma}

These estimates
are proved, for example, in Vekua \cite[Chapter I.6]{Vekua:1962} or Astala, Iwaniec, and Martin \cite[\S 4.3]{AIM:2009}. 

\begin{remark}
\label{rem:I1}
The estimates \eqref{P1}--\eqref{P3} are valid for the integral
$$ I_1(f) (x) = \int \frac{1}{|z-\xi|} f(\xi) \, dm(\xi). $$
The estimate \eqref{P1} in this instance is the case $\alpha=1$, $n=2$ of the 
Hardy-Littlewood-Sobolev inequality $\norm{|x|^{-\alpha} * f}{q} \lesssim_{\,p}\norm{f}{p}$ for $\dfrac{1}{q} = \dfrac{\alpha}{n} + \dfrac{1}{p} - 1$.
\end{remark}

The following lemma is standard (see for example, \cite[Lemma 2.2]{Perry:2016}).

\begin{lemma}
\label{lemma:P2}Suppose that $p\in(2,\infty)$, that $u\in L^{p}(\bbR%
^{2})$, that $f\in L^{2p/(p+2)}({\bbC})$, and that $\overline
{\partial}u=f$ in distribution sense. Then $u=\Cauchy f$. Conversely, if $f\in
L^{2p/(p+2)}({\bbC})$ and $u=\Cauchy f$, then $\dbar u=f$ in
distribution sense.
\end{lemma}

Similarly, to solve the equation $\dee u = f$, we introduce the operator
\begin{equation}
\label{Pbar}
\left[ \CauchyBar f \right](z) = \frac{1}{\pi} \int \frac{1}{\zbar - \zetabar} \,  f(\zeta) \, dm(\zeta)
\end{equation}
which obeys analogous estimates. We don't state the obvious analogue of Lemmas \ref{lemma:P1} and \ref{lemma:P2} for the operator $\CauchyBar$.

The following formulas will help find a basis for the nullspace of the integral operator that describes the one-soliton solution for fDSII. Let
\begin{equation}
\label{rho}
\rho(z) = \left(1+|z|^2 \right)^{1/2}.
\end{equation}
From the trivial identities
$$ \dbar\left(\zbar \rho^{-2}\right) = \rho^{-4}, \quad
    \dee\left(\rho^{-2}\right) = -\zbar \rho^{-4}, $$
their complex conjugates, and Lemma \ref{lemma:P2}, we easily deduce
\begin{subequations}
\label{PPbar.int}
\begin{alignat}{2}
\Cauchy\left[ \rho^{-4} \right] 
		&= \zbar \rho^{-2} \qquad
	& \CauchyBar\left[ \rho^{-4} \right] 
		&=z\rho^{-2}\\
\CauchyBar\left[ \zbar \rho^{-4} \right] 
		&= -\rho^{-2} \qquad
	& \Cauchy\left[ z\rho^{-4} \right] 
		&= -\rho^{-2}
\end{alignat}
\end{subequations}

\subsection{Beurling Operator}

Let
$$ 
\left(\calS f \right)(z) = -\frac{1}{\pi} 
	\lim_{\eps \darr 0} \
		\left( \int_{|z-z'|>\eps} \frac{1}{(z-z')^2} \, f(z') \ dm(z') \right)
$$
initially defined on $C_0^\infty(\bbC)$. The operator $\calS$ is the \emph{Beurling transform}. See for example \cite[\S 4.5.2]{AIM:2009} for proofs and discussion

\begin{lemma}
\label{lemma:S.bounded}
Suppose that $p \in (1,\infty)$. The operator $\calS$ extends to a bounded linear operator from $L^p(\bbC)$ to itself , unitary if $p=2$. Moreover, if $\nabla f \in L^p(\bbC)$ for some $p \in (1,\infty)$, $\calS(\dbar \phi) = \dee \phi$.
\end{lemma}

\subsection{Mixed $L^p$ Spaces}
We review some basic facts about mixed $L^p$ spaces; a standard reference
is the paper of Benedek and Panzone \cite{BP:1961}. Suppose that $a$ is a
measurable function on $\mathbb{C}\times\mathbb{C}$. For $1<p,q<\infty$, we
define%
\[
\left\Vert a\right\Vert _{L^{p}\left(  L^{q}\right)  }=\left(  \int
_{\mathbb{C}}\left(  \int_{\mathbb{C}}\left\vert a(z,w)\right\vert
^{q}~dw\right)  ^{p/q}~dz\right)  ^{1/p}%
\]
and%
\[
\left\Vert a\right\Vert _{L^{p}\left(  L^{\infty}\right)  }=\left(
\int_{\mathbb{C}}\left\Vert a(z,~\cdot~)\right\Vert _{\infty}^{p}~dz\right)
^{1/p}.
\]
We denote by $L^{p}\left(  L^{q}\right)  $ (resp. $L^{p}\left(  L^{\infty
}\right)  $) the Banach space of complex-valued measurable functions $a$ on
$\mathbb{C}\times\mathbb{C}$ with $\left\Vert a\right\Vert _{L^{p}\left(
L^{q}\right)  }$ (resp. $\left\Vert a\right\Vert _{L^{p}\left(  L^{\infty
}\right)  }$) finite. Note that $L^{1}(L^{1})$ is the space $L^{1}%
(\mathbb{C}\times\mathbb{C})$. We have H\"{o}lder's inequality%
\[
\left\Vert ab\right\Vert _{L^{1}}\leq\left\Vert a\right\Vert _{L^{p}\left(
L^{q}\right)  }\left\Vert b\right\Vert _{L^{p^{\prime}}\left(  L^{q^{\prime}%
}\right)  }%
\]
and%
\begin{equation}
\left\Vert a\right\Vert _{L^{p}\left(  L^{q}\right)  }=\sup_{\left\Vert
g\right\Vert _{L^{p^{\prime}}(  L^{q^{\prime}}) = 1 }}\left\vert
\int_{\mathbb{C}\times\mathbb{C}}g(z,w)a(z,w)~dw~dz\right\vert .
\label{holder}%
\end{equation}
We denote by $L^{p}\cap L^{p^{\prime}}\left(  L^{q}\right)  $ the space
$L^{p}\left(  L^{q}\right)  \cap L^{p^{\prime}}\left(  L^{q}\right)  $ with
norm%
\[
\left\Vert a\right\Vert _{L^{p}\cap L^{p^{\prime}}\left(  L^{q}\right)
}=\left\Vert a\right\Vert _{L^{p}\left(  L^{q}\right)  }+\left\Vert
a\right\Vert _{L^{p\prime}\left(  L^{q}\right)  },
\]
while $L^{p}\left(  L^{q}\cap L^{q^{\prime}}\right)  $ denotes the space
$L^{p}\left(  L^{q}\right)  \cap L^{p}\left(  L^{q^{\prime}}\right)  $ with
norm%
\[
\left\Vert a\right\Vert _{L^{p}\left(  L^{q}\cap L^{q^{\prime}}\right)
}=\left\Vert a\right\Vert _{L^{p}\left(  L^{q}\right)  }+\left\Vert
a\right\Vert _{L^{p}\left(  L^{q^{\prime}}\right)  }.
\]

If $g\in\mathcal{S}\left(  \mathbb{C}\times\mathbb{C}\right)  $ and
$g_{\Delta}(\zeta)=g\left(  \zeta,\zeta\right)  $, then
\begin{equation}
\left\Vert g_{\Delta}\right\Vert _{L^{p}}\leq\left\Vert g\right\Vert
_{L^{p}\left(  L^{\infty}\right)  }\text{.} \label{diag}%
\end{equation}
For any complex-valued measurable function $a$ on $\mathbb{C}\times\mathbb{C}%
$, we denote by $a^{\ast}$ the measurable function%
\[
a^{\ast}(z,w)=\overline{a(w,z)}.
\]

\subsection{Perturbation Theory}
\label{subsec:Kato-Rellich}

In this subsection we recall some elements of Kato-Rellich perturbation theory as they apply to the perturbation of soliton solutions studied in \S \ref{sec:one-soliton}-\ref{sec:perturb}. We consider a norm-continuous mapping $t \mapsto A(t)$ from an open neighborhood  $U$ of $0$ in $\bbR^n$ to the compact operators on a Banach space $X$. 

Let us suppose that $A(0)$ has the isolated eigenvalue $1$. There is a $\delta>0$ so that the circle $|\lambda-1|=\delta$ divides the spectrum of $A(0)$ into disjoint sets, and there is an $\eps>0$ so that for all $t$ with $|t|<\eps$, the same circle divides the spectrum of $A(t)$ into two parts. We may form the projections
\begin{equation}
\label{P(t)}
P(t) = -\frac{1}{2\pi i} \oint_{|\lambda-1|=\delta} (A(t)-\lambda I)^{-1} \, d\lambda 
\end{equation}
and
\begin{equation}
\label{Q(t)}
Q(t) = I - P(t). 
\end{equation}
The projections $P(t)$ and $Q(t)$ are continuous operator-valued functions for $t$ with $|t|<\eps$. For each fixed $t$, $P(t)$ and $Q(t)$ commute with $A(t)$. Since $P(t)^2 = P(t)$ it follows that $P(t)Q(t)=0$. By decreasing $\eps$ if necessary we may assume that $\| P(t)-P(0) \|<1/2$ for all $t$ with $|t|<\eps$, so that $\dim P(t) = \dim P(0)$.

\begin{lemma}
\label{lemma:APQ} For all $t$ sufficiently small, 
the operator $(I-A(t))$ is invertible if and only the operator $(I-P(t)A(t)P(t))$ is invertible.
\end{lemma}

\begin{proof}
Let us write $A$, $P$, $Q$ for $A(t)$, $P(t)$, and $Q(t)$.
The operator $QAQ$ has no spectrum in the region $|\lambda-1|<\delta$ so the inverse $(I-QAQ)^{-1}$ exists for all $t$ with $|t|<\eps$. Computing
$$ (I-A)(I-QAQ)^{-1} = (I-PAP-QAQ)(I-QAQ)^{-1} = I-PAP $$
we see that $(I-A)^{-1}$ exists if and only if $(I-PAP)^{-1}$ exists.
\end{proof}

We can make a further reduction using an observation of Sz.-Nagy \cite{SzNagy:1946} already used by Gadyl'shin and Kiselev in their analysis of the one-soliton perturbation. Write $P_0$ for $P(0)$.

\begin{lemma}
\label{lemma:APQW}
For sufficiently small $t$, there is an invertible operator $V(t)$ so that 
$PAP$ is similar to $P_0V(t)^{-1}A(t)V(t)P_0$, and 
$I-\lambda PAP$ is invertible if and only if $I-\lambda P_0 V(t)^{-1} A (t) V(t) P_0$ is invertible.
\end{lemma}

\begin{proof}
We set
\begin{equation}
\label{V}
V= \left(I-(P-P_0)^2 \right)^{-1/2} \left[ PP_0 + (1-P)(1-P_0) \right].
\end{equation}
It is not difficult to see that, if $\| P - P_0 \| < 1/2$,
\begin{equation}
\label{V.est}
\| V - I \| \leq 2 \| P - P_0 \| 
\end{equation}
and that $PAP$ is similar to 
\begin{equation}
\label{A.similar}
P_0 V^{-1} A V P_0 
\end{equation}
(see \cite{RS:1978}, notes to \S XII.2 and problem 19 of chapter XII, and see also the classic paper of Sz.-Nagy \cite{SzNagy:1946}).  It now follows that $(\lambda I-A)$ is invertible if and only if $\lambda I-P_0V^{-1}AVP_0$ is invertible.
\end{proof} 

\subsection{Eigenvalue Multiplicities}

In what follows we denote by $\Det(I + \, \dotarg \, )$ a generalized determinant defined on an algebra $\calE$ of compact operators on a Banach space $X$, 
having the following properties:

\smallskip

(i) $(I+A)$ is invertible if and only if $\Det(I+A) \neq 0$, and

\smallskip

(ii) $\Det(I+F) = \det(I+F) e^{\Tr F}$ for finite-rank operators $F$.

\smallskip

In applications, $X=L^p$, $\calE$ is the Mikhlin-Itskovich algebra, and
$\Det(I + \, \dotarg \,)$ is the generalized determinant described in Appendix \ref{sec:GGK}. 

In this subsection, we prove:
\begin{lemma}
\label{lemma:multiplicities}
Suppose that $A(\kappa)$ is a $C^1$ compact operator-valued function defined on an open neighborhood of $0$ in $\bbC$. Suppose further that the eigenvalue $\lambda=1$ of $A(0)$ is semisimple, and that $\Det(I-A(\kappa))=c|\kappa|^m(1+o(1))$ as $k \rarr 0$. Then $\ker (I-A(0))$ has dimension at most $m$.
\end{lemma}

\begin{proof} 
The operator $A(0)$ is compact so $\ker(I-A(0))$ is at most finite-dimensional. Moreover, for $\kappa$ small, there is a $\delta>0$ so that the circle $|\lambda-1|=\delta$ divided the spectrum of $A(\kappa)$ into two disjoint parts. Let $P$, $Q$, $P_0$, $Q_0$, $V$ be as in \S \ref{subsec:Kato-Rellich} above. We analyze $\Det(I-A(\kappa))$ for $\kappa$ small by splitting $I-A=I-PAP-QAQ$. Using the determinant formula \eqref{Det.Multiply}, we factor
$$(I-PAP-QAQ)=(I-PAP) (I-QAQ)$$ (since $PQ=QP=0)$ and conclude from \eqref{Det.Multiply} that
$$ \Det(I-A)=\Det(I-PAP)\Det(I-QAQ) $$
since
$PAP \cdot QAQ =0$. Moreover, from the discussion in the previous section, $PAP$ is similar to $P_0 V^{-1}A V P_0$ so
$$\Det(I-A) = \Det\left(I-P_0 V^{-1} A V P_0 \right) \Det(I-QAQ).$$
The second factor is nonvanishing and has a finite nonzero limit as $\kappa \rarr 0$, so 
the leading asymptotics are determined by the first factor. Since $A(\kappa)-A(0) = \bigO{|k|}$ in operator norm as $\kappa \rarr 0$, it follows that, also,
$\| P-P_0 \| = \bigO{|k|}$, $V-I = \bigO{|k|}$. Since $A(0)$ has semisimple eigenvalues and $P_0 V^{-1} A V P_0$ is a rank $N$ operator, we may choose a basis of  eigenvectors $\{ \psi_i \}_{i=1}^N$ for $A(0)$ in $X$ and a dual basis $\{ \chi_i \}_{i=1}^N$ in $X^*$ so that
$$ \left\langle \chi_i, \psi_j \right\rangle = \delta_{ij} $$
where $\left\langle \dotarg, \dotarg \right\rangle$ is the usual dual pairing. It follows that 
$$ P_0 = \sum_{i=1}^N \left\langle \chi_i, \dotarg \right\rangle  \psi_i .$$
Hence $\Det(I-P_0 V^{-1} A V P_0)$ is, up to strictly nonzero factors, the determinant of the $N \times N$ matrix $M$ with 
$$ M_{ij} = \left\langle \chi_i, \left[I - V^{-1} A V\right] \psi_j\right\rangle. $$
Since $V^{-1} A(\kappa) V = A(0)+ \bigO{|\kappa|}$, it follows that $M_{ij} =\bigO{|\kappa|}$ and so
$$ \det M = \bigO{|\kappa|^N}. $$
Hence $m \geq N$. 
\end{proof}

\begin{remark}
The conclusion of Lemma \ref{lemma:multiplicities} is false if the eigenvalue $\lambda=1$ is not semisimple. To see this, consider the matrix
$$ 
A(\eps) = 
\left( 
		\begin{array}{ccc} 
			1 & 1 & \eps \\ \eps & 1 & 1  \\ \eps & \eps &1
		\end{array}
	\right)
$$
for which  $\det(I-A(\eps)) = \eps^3+\eps$, so $N=3$ but $m=1$).
\end{remark}

\section{A Fredholm Determinant for Direct Scattering}
\label{sec:direct}

In this section we characterize the exceptional set $Z$ as the zero set of a renormalized Fredholm determinant associated to the scattering problem \eqref{CGO.m}. 

\subsection{Reduction by Symmetries}

For $p>2$, $u \in L^2$, and $k \in \bbC$, define an operator $S_{k,u} \in \calB(L^p)$ by
\begin{equation}
\label{SKU}
S_{k,u} h = 
		\frac{1}{4}\Cauchy \left( u e_{-k} \CauchyBar \left( e_k \ubar h \right) \right).
\end{equation}
In this subsection, we prove:

\begin{proposition}
\label{prop:symmetry} Suppose that $u \in L^2(\bbC)$ and $p>2$. Then $k \in \bbC$ is an exceptional point  for the problem \eqref{CGO.m} if and only if $\ker_{L^p} (I - S_{k,u})$ is nontrival.
\end{proposition}

Recall that $k$ is an exceptional point for \eqref{CGO.m} if the problem \eqref{CGO.m.pde} has
a nontrivial solution with $M(z,k) \in L^p(\bbC)$. We reduce to 
a single integral equation involving the integral operator \eqref{SKU} in several steps.

\begin{lemma}
\label{lemma:CGO.sym}
Fix $p>2$. Suppose that $u \in L^2(\bbC)$ and $k \in \bbC \setminus Z$. Then, the unique solution $M$ of \eqref{CGO.m}
with 
$$M-\twomat{1}{0}{0}{1} \in L^p(\bbC)$$ takes the form
\begin{equation}
\label{m.sym}
M(z,k) = \twomat{m_{1}(z,k)}{-\overline{m_{2}(z,k)}}{m_{2}(z,k)}{\overline{ m_{1}(z,k)}}
\end{equation}
where
\begin{subequations}
\label{CGO.m.sym}
\begin{align}
\label{CGO.m.sym1}
\dbar m_{1}	&=	\frac{1}{2} u m_{2} \\
\label{CGO.m.sym2}
(\dee + ik) m_{2} &= - \frac{1}{2} \ubar m_{1} \\
\label{CGO.m.sym.bc}
 (m_{1}(z,k)-1, m_{2}(z,k)) &\in L^p(\bbC)
\end{align}
\end{subequations}
\end{lemma}

\begin{proof}
A straightforward computation shows that the function \eqref{m.sym} solves \eqref{CGO.m}, so the result now 
follows by unicity.
\end{proof}

Thus, to compute the exceptional set, it suffices to study the system \eqref{CGO.m.sym}. By Lemma \ref{lemma:P2}, we can reduce the system \eqref{CGO.m.sym} to a system of integral equations
using the Cauchy transform.  In what follows, the condition $u \in L^{2p/(p+2)}(\bbC)$ insures that expressions such as $\CauchyBar (e_k \ubar)$ define functions in $L^p(\bbC)$. 

\begin{lemma}
\label{lemma:CGO.int}
Fix $p>2$, $u \in L^2(\bbC) \cap L^{2p/(p+2)}(\bbC)$, and $k \in \bbC \setminus Z$. A vector-valued function $\bfm= (m_1,m_2)$ with $(m_1 -1, m_2) \in  L^p(\bbC)$ solves \eqref{CGO.m.sym} if and only if 
\begin{subequations}
\label{CGO.m.int}
\begin{align}
\label{CGO.m.int1}
m_1	&=	1	+	\Cauchy\left( \frac{1}{2} u m_2 \right) \\
\label{CGO.m.int2}
m_2	&=			-\frac{1}{2} e_{-k} \Cauchybar \left( e_k \ubar m_1 \right)
\end{align}
\end{subequations}
\end{lemma}

Finally, we can iterate to a scalar integral equation \eqref{CGO.scal}:

\begin{lemma}
\label{lemma:CGO.scal}
Fix $p>2$, $u \in L^2(\bbC) \cap L^{2p/(p+2)}(\bbC)$, and $k \in \bbC \setminus Z$. The vector-valued function $\bfm$ 
with $(m_1 -1, m_2) \in  L^p(\bbC)$ solves \eqref{CGO.m.int} if and only if 
\begin{equation}
\label{CGO.scal}
m_1	=	 1	- \frac{1}{4} \Cauchy\left(  u e_{-k} \Cauchybar \left( e_k \ubar m_1 \right) \right) 
\end{equation}
and
\begin{equation}
\label{CGO.scal.m2}
m_2 = -\frac{1}{2} e_{-k} \Cauchybar \left( e_k \ubar m_1 \right). 
\end{equation}
\end{lemma}

We omit the (easy) proofs of Lemmas \ref{lemma:CGO.int} and \ref{lemma:CGO.scal}. The compositions with $\Cauchy$ and $\Cauchybar$ make sense since $\Cauchy, \Cauchybar: L^{2p/(p+2)} \rarr L^p$ by \eqref{P1} and $u f$ in $L^{2p/(p+2)}$ by H\"{o}lder's inequality provided  $u \in L^2$ and $f \in L^p$. 

We now complete the reduction to a scalar problem. We will sometimes
decompose 
$$ S_{k,u} = W_{k,u} \circ V_{k,u} $$
where
$$ W_{k,u} h = \frac{1}{2} \Cauchy\left( u e_{-k} h \right)$$
and
$$ V_{k,u} h =  \frac{1}{2} \Cauchybar\left(  e_k \ubar h \right)$$

\begin{lemma}
\label{lemma:S}
Fix $p>2$ and $k \in \bbC$. Suppose that $u \in L^2 \cap L^{2p/(p+2)}$. Then, the operator $S_{k,u}$ is compact  as an operator from $L^p(\bbC)$ to itself. 
\end{lemma}

\begin{proof}
The equation \eqref{CGO.scal} is equivalent to 
$$ m_1 - 1 = S_{k,u} 1 - S_{k,u} (m_1-1). $$
It follows from Lemma \ref{lemma:P1} and its analogue for $\Cauchybar$ that $S_{k,u}$ is a 
bounded operator on $L^p$ provided $u \in L^2 \cap L^{2p/(p+2)}$. Moreover, since $S_{k,u}$ is bilinear in $u$,
it is easy to see that the map $L^p \cap L^{2p/(p+2)} \ni u \mapsto S_{k,u} \in \calB(L^p)$ is continuous. Hence, 
to prove that $S_{k,u}$ is compact, it suffices to do so for $u \in C_0^\infty(\bbC)$ and appeal to density. We can
argue as in the first paragraph of \cite[proof of Lemma 3.1]{Perry:2016} that $W_{k,u}$ is compact, while $V_{k,u}$ is bounded by Lemma \ref{lemma:P1} again. Hence $S_{k,u}$ is compact.
\end{proof}

\begin{proof}[Proof of Proposition \ref{prop:symmetry}]
It follows from Lemmas \ref{lemma:CGO.sym} -- \ref{lemma:CGO.scal} and the 
Fredholm alternative that the problem \eqref{CGO.m} has a unique solution 
if and only if $\ker(I-S_{k,u})$ is trivial.
\end{proof}

\subsection{Renormalized Determinant}

We'll now define and study a renormalized determinant of $I-S_{k,u}$. In Proposition \ref{prop:brown}, it is shown that the operator $S_{k,u}$ belongs to the Miklhlin-Itskovich algebra $\calE_p$ provided $p>2$ and $u \in L^t(\bbC) \cap L^{t'}(\bbC)$ where
\begin{equation}
\label{tp}
\frac{1}{2}+\frac{1}{p} < \frac{1}{t}, \quad \frac{1}{p} + \frac{1}{t} > 1.
\end{equation}

\begin{definition}
\label{def:admissible}
We say that $(p,t)$ is an \emph{admissible pair}
if $p>2$, $t \in [1,2)$, and eqref{tp} holds.
\end{definition}

\begin{remark}
\label{rem:admissable}
The two constraints
\eqref{tp} together with $p>2$ and $t>1$ imply that $(1/p,1/t)$ belong to the interior of the triangle with vertices
$(0,1)$, $(1/2,1)$ and $(1/4,3/4)$ in the $(1/p,1/t)$-plane (see Figure \ref{fig:tp} in Appendix 
\ref{app:brown}). If $(p,t)$ is an admissable pair and $u \in L^{t}(\bbC) \cap L^{t'}(\bbC)$, it
is easy to see that $u \in L^{2p/(p+2)}$ since, by \eqref{tp}, the inequalities
$$ \frac{1}{t'} < \frac{1}{2} + \frac{1}{p} < \frac{1}{t} $$
hold.
\end{remark}

For an admissible pair $(p,t)$, the renormalized determinant of Theorem \ref{thm:GKK} 
\begin{equation}
\label{D}
 D(k,u) = \Det(I-S_{k,u}) 
 \end{equation}
is a well-defined, bounded continuous function of $(k,u)$ with $D(k,u) \rarr 1$ as $k \rarr \infty$ and $D(k,u)$ is continuous in $u \in L^{t}(\bbC) \cap L^{t'}(\bbC)$ uniformly in 
$k \in \bbC$.  

We will define the determinant in  Banach space of potentials large enough to include $C_0^\infty(\bbC)$ perturbations of the soliton solution \eqref{APP}, and sufficiently restrictive that the $\dbar$ equation stated in Theorem \ref{thm:dbar} holds.  For $\alpha \in (1/2,1)$ let 
$$X_\alpha= W^{1,2}(\bbC) \cap L^{2,\alpha}(\bbC)$$ 
where $W^{1,2}(\bbC)$ consists of $L^2$ functions with one weak derivative in $L^2$ and 
$$L^{2,\alpha}(\bbC) = \left\{ f \in L^2(\bbC): (1+|z|^2)^{\alpha/2} f \in L^2(\bbC) \right\}.$$
Note that $X_1$ is the space $H^{1,1}(\bbC)$ considered in \cite{Perry:2016}. We need $\alpha<1$ to include the soliton solution \eqref{APP}, but $\alpha>1/2$ for later estimates. It is easy to see that if $u \in X_\alpha$, then 
$u \in L^{q}(\bbC)$ for $q \in \left( \dfrac{2}{1+\alpha},\infty\right)$, so that
$u \in L^{t}(\bbC) \cap L^{t'}(\bbC)$ for $\dfrac{1}{1+\alpha} < t< 2$. To find admissible $(p,t)$ with $t \in (1,4/3)$, we require $\alpha > 1/2$. 

\medskip

We also note:

\begin{lemma}
\label{lemma:compact.embedding}
The inclusion $X_\alpha \rarr L^2(\bbC)$ is a compact embedding.
\end{lemma}

\begin{proof}
For any bounded set $\Omega \subset \bbC$ with smooth boundary, the compact embedding $W^{1,2}(\Omega) \hookrightarrow L^2(\Omega)$ holds. Let $\chi \in C_0^\infty(\bbC)$ with $\chi(z) = 1 \text{ if } |z| \leq 1$ and $\chi(z) = 0$ if $|z| \geq 2$. Let
$\chi_R(z) = \chi(z/R)$. For each $R>0$, the map $f \mapsto \chi_R f$ is compact from $W^{1,2}(\bbC)$ into $L^2(\Omega)$. For $u \in X_\alpha$, 
$ \norm{(1-\chi_R)u}{L^2(\bbC)} \leq R^{-\alpha} \norm{u}{L^{2,\alpha}(\bbC)}$ so
taking $R \rarr \infty$ we see that the embedding into $L^2(\bbC)$ is compact.
\end{proof}

From Proposition \ref{prop:brown} and the remarks above, we have:

\begin{proposition}
Suppose that $u \in X_\alpha$ for some $\alpha \in (1/2,1)$. Then the renormalized determinant
$$D(k,u) = \Det(I- S_{k,u})$$
is continuous in $k$ and $u$, and satisfies the asymptotic condition
$$ \lim_{|k| \rarr \infty} D(k,u) = 1.$$
\end{proposition}

Clearly, $Z = \{ k \in \bbC: D(k)  = 0 \}$. As an immediate corollary, we have:

\begin{corollary}
\label{coro:exbd}
Suppose that $u \in X_\alpha$ for some $\alpha \in (1/2,1)$. Then the exceptional set $Z$ is closed and bounded.
\end{corollary}

\begin{remark}
By Remark \ref{rem:Det.common}, if $S_{k,u}$ belongs to $\calE_p$ and $\calE_{q}$ for
distinct $p$ and $q$, the determinants in $\calE_p$ and $\calE_q$ coincide. 
\end{remark}

\subsection{A $\dbar$-Equation for the Determinant}
We will now derive a $\dbar$ equation for $D(k)$ in terms of scattering data for $u$, defined as follows. We first compute for $u \in C_0^\infty(\bbC)$ and then use continuity to pass to $u \in X_\alpha$. 

For $\delta >0$, let 
$$ \Omega_\delta = \left\{ k \in \bbC: \dist(k,Z) > \delta \right\}. $$
By Corollary \ref{coro:exbd}, $\Omega_\delta$ is an unbounded open set that contains a neighborhood of infinity.
On this set, the solution of \eqref{CGO.m.int} is unique, and we define scattering data $r$ and $s$, functions of
$k \in \Omega_\delta$, by the asymptotic formulas
\begin{align}
\label{data.s.lim}
s(k) &= 2  \lim_{|z| \rarr \infty} z \left(m_1(z,k)-1 \right) \\
\label{data.r.lim}
r(k) &= -2  \lim_{|z| \rarr \infty} \left(e_{-k}(z) z \overline{m_2(z,k)} \right)
\end{align}
The existence of these limits is a simple consequence of the formula
\begin{equation}
\label{Cauchy.lim}
\lim_{|z| \rarr \infty} z \, \Cauchy\left[ f \right](z) = \frac{1}{\pi} \int f(z) \, dm(z) 
\end{equation}
valid if $f \in L^1(\bbC) \cap L^p(\bbC)$ for some $p> 2$ (see \eqref{P3}). From \eqref{CGO.m.int} and \eqref{Cauchy.lim}, we deduce that
\begin{align}
\label{data.s.int}
s(k)	&=	\frac{1}{\pi} \int u(z) m_2(z,k) \, dm(z) 	\\
\label{data.rint}
r(k)	&=	-\frac{1}{\pi} \int e_{-k}(z) u(z) \overline{m_1(z,k)} \, dm(z)
\end{align}

In this section, we will prove:

\begin{theorem}
\label{thm:dbar}
Suppose that $u \in X_\alpha$ and $\delta>0$.
Then $D(k)$ defined by \eqref{D} obeys the $\dbar$-equation
\begin{equation}
\label{D.dbar}
\dbar \log D(k) = \frac{i}{2}\overline{s(k)}-c(k)
\end{equation}
for all $k \in \Omega_\delta$, where
\begin{equation}
\label{c.bis}
c(k) = -\frac{i}{4\pi^2} \iint \frac{e_{-k}(w)u(w) e_k(z) \overline{u(z)}}{z-w} \, dm(w) \, dm(z).
\end{equation}
\end{theorem}

\begin{remark}
\label{rem:c}
Differentiating \eqref{c.bis} with respect to $k$ and using the analogue of Lemma \ref{lemma:P2} for the $\dee$-operator, we conclude that
$$ c(k) = \frac{1}{4\pi} \int \frac{1}{\kbar - \zetabar} \left| \left(\calF u\right)(-\zeta) \right|^2 \, dm(\zeta). $$
\end{remark}

We begin by considering $u \in C_0^\infty(\bbC)$.

\begin{proposition}
\label{prop:D.dbar}
Suppose that $u \in C_0^\infty(\bbC)$ and $\delta>0$.
Then, the conclusion of Theorem \ref{thm:dbar} holds.
\end{proposition}

\begin{proof}
From
\[
S_{k,u}=\frac{1}{4}\dbar^{-1}
	\left(  
		e_{-k}u\partial^{-1}
			\left(  
				e_{k}\overline{u}~\cdot~
			\right)  
	\right)
\]
we easily compute that
\[
\left[  \dbar_{k},S_{k,u}\right]  =-\frac{1}{4}\overline
{\partial}^{-1}\left(  e_{-k}u~i\left[  \overline{z},\partial^{-1}\right]
~\left(  e_{k}\overline{u}~\cdot\right)  \right)
\]
which is a rank-one integral operator with integral kernel%
\begin{equation}
K(z,z^{\prime};k)=-\frac{i}{4\pi^{2}}\left[  \int\frac{1}{z-w}e_{-k}%
(w)u(w)~dm(w)\right]  ~e_{k}(z^{\prime})\overline{u}(z^{\prime}).~
\label{eq:K}%
\end{equation}
This rank-one operator belongs to $\calE _{p}$ provided $u\in
L^{p^{\prime}}\cap L^{2p/(p+2)}$

We apply Lemma \ref{lemma:Det.dot}. Compute
\[
\dbar_{k}\log\Det\left(  I- S_{k,u}\right)  =\operatorname*{Tr}\left(  \left(  -1 \right)  \left(
I- S_{k,u}\right)  ^{-1}\dbar_{k}S_{k,u}\right)
+\operatorname*{Tr}\left(  \dbar_{k}S_{k,u}\right)  .
\]
Each of these terms is the trace of a rank-one operator. The first has
integral kernel
\[
-\frac{1}{2\pi}\left[  \left(  I- S_{k,u}\right)  ^{-1}\left(
\frac{i}{2}\dbar^{-1}\left(  e_{-k}u\right)  \right)  \right]
\left(  z\right)  e_{k}(z^{\prime})\overline{u}(z^{\prime})=-\frac{i }{2\pi}%
m_{2}(z,k)e_{k}(z^{\prime})\overline{u(z^{\prime})}%
\]
where we used the fact that
\begin{align*}
\left(  I- S_{k,u}\right)  ^{-1}\left(  \frac{1}{2}%
\dbar^{-1}\left(  e_{k}u\right)  \right)   &  = \left(  I-
S_{k,u}\right)  ^{-1}T_{k}1\\
&  = m_{2}.
\end{align*}
The second term has integral kernel $ K$ where $K$ is given by
(\ref{eq:K}) so that, altogether,%
\begin{align*}
(-1)\Tr
	\left(  
			\left(  I- S_{k,u}\right) ^{-1}
			\dbar_{k}S_{k,u}
	\right)
		&  
+  \Tr
	\left(
			\dbar_{k}S_{k,u}
	\right)   \\
&	 =\frac{i}{2\pi}
	\int
			e_{k}(z)\overline{u(z)}m_{2}(z,k)
	\,dm(z)\\
&  \quad +\frac{i}{4\pi^{2}}
	\int\int
			\frac{e_{-k}(w)u(w)e_{k}(z)
			\overline{u(z)}}{z-w}
	~dm(w)~dm(z)~
\end{align*}
which gives the claimed formula by (\ref{data.s.int}). 
\end{proof}

Now we would like to prove that $D(k)$ solves the same $\dbar$-equation weakly if $u \in X_\alpha$ for some $\alpha \in (1/2,1)$. The following proposition gives a ``to-do list.''

\begin{proposition}
\label{prop:todo}
Fix $\alpha \in (1/2,1)$ and $\delta>0$. Suppose that the map $u \mapsto m_2(\dotarg,k;u)$ defined by 
\eqref{CGO.scal}--\eqref{CGO.scal.m2} has the following property: if $u \in X_\alpha$ 
and $\{ u_n \}$ is a sequence from $C_0^\infty(\bbC)$ converging to $u$,  
and if $Z_n$ (resp. $Z$) denotes the exceptional set of $u_n$ (resp. $u$), then
\begin{itemize}
\item[(i)]		For the given $\delta>0$ and all $n$ sufficiently large depending on $\delta$, the condition $\dist(k,Z) > 2\delta$ implies that $\dist(k,Z_n)>\delta$, and
\item[(ii)]	$m_2(\dotarg, \diamond ;u_n) \rarr m_2(\dotarg,\diamond,u)$ in $L^\infty(\Omega_\delta ; L^p(\bbC))$.
\end{itemize}
Then  \eqref{D.dbar} holds weakly on $\Omega_\delta$.
\end{proposition}

\begin{proof}
Assuming the continuity, let $u \in X_\alpha$ and let $\{ u_n \}$ be a sequence from $C_0^\infty(\bbC)$ 
with $u_n \rarr u$ in $X_\alpha$.  Since $\alpha > 1/2$, $u \in L^t \cap L^{t'}$ for some $t \in (1,4/3)$ 
and there is a $p$ with $\widetilde{t} < p < t'$ so that $(p,t)$ is an admissible pair of exponents. 
A short computation shows that $p' \in (t,t')$ so that, also, $u \in L^{p'}(\bbC)$ with norm bounded by 
$\norm{u}{X_\alpha}$. Denoting by $s_n$ the scattering data corresponding to $u_n$, 
we have from \eqref{data.s.int}
that
\begin{align*}
\norm{s_n - s}{L^\infty(\Omega_\delta)} 
	&\leq \norm{u-u_n}{p'} \norm{m_2(\dotarg,\diamond;u_n)}{L^\infty(\Omega_\delta;L^p(\bbC))} \\
	&\quad
		+ \norm{u_n}{p'}\norm{m_2(\dotarg,\diamond;u_n)-m_2(\dotarg,\diamond;u)}{L^\infty(\Omega_\delta;L^p(\bbC))}
\end{align*}
so that, if the hypothesis holds, $s_n \rarr s$ in $L^\infty(\Omega_\delta)$ as $n \rarr \infty$. 

Next, let 
$$ c_n = -\frac{i}{4\pi^2} \int \, \int \frac{e_{-k}(w) u_n(w) e_k(z) \overline{u_n(z)}}{z-w} \, dm(w) \, dm(z). $$
From \eqref{P1} we easily see that
$$ |c_n| \lesssim \norm{u_n}{p} \norm{u_n}{2p/(p+2)} $$
so by bilinearity
$$ |c_n - c|  \lesssim \norm{u_n - u}{p} \norm{u_n}{2p/(p+2)} + \norm{u}{p} \norm{u_n - u}{2p/(p+2)}. $$
The $X_\alpha$ norm dominates the $L^p$ and $L^{2p/(p+2)}$ norms so that
$\norm{c_n-c}{L^\infty(\bbC} \rarr 0$ as $n \rarr \infty$.

Finally, let $\varphi \in C_0^\infty(\Omega_\delta)$. For sufficiently large $n$, $\Det(k;u_n)$ is defined for all $k \in \Omega_\delta$ and we may compute
\begin{align*}
\int_{\Omega_\delta} &
		\left[ 
			(-\dbar \varphi) \log D(k,u) 
				- \varphi 
					\left( 
						\frac{i}{2}\overline{s(k)} - c(k)
		 			\right) \
		 \right] \, dm(k)\\
&=	
\lim_{n \rarr \infty} 
\int_{\Omega_\delta} 
		\left[ 
			(-\dbar \varphi) D(k,u_n) 
				- \varphi 
					\left( 
						\frac{i}{2}\overline{s_n(k)} - c_n(k)
		 			\right) \
		 \right] \, dm(k)\\
&=
\lim_{n \rarr \infty} 
\int_{\Omega_\delta} 
		\varphi 
		\left[ 
			\dbar \log D(k,u_n) 
				- 
					\left( 
						\frac{i}{2}\overline{s_n(k)} - c_n(k)
		 			\right) \
		 \right] \, dm(k)\\
&=0.
\end{align*}

\end{proof}

It remains to show that the hypothesis of Proposition \ref{prop:todo} holds. First:

\begin{lemma}
\label{lemma:Z}
Fix $\delta>0$. Suppose that $u \in X_\alpha$ for some $\alpha \in (1/2,1)$, and 
that  $\{ u_n \}$ is a sequence from $X_\alpha$ with $u_n \rarr u$. Finally, let $Z_n$ and $Z$
be the respective exceptional sets for $u_n$ and $u$. There is an $N$ 
so that for any $n>N$, $\dist(z,Z_n) > \delta$ provided $\dist(z,Z)>2\delta$.
\end{lemma}

\begin{proof}
Since $u_n \rarr u$ in $X_\alpha$, it follows that there is an admissible pair of exponents $(p,t)$ so
that $u_n \rarr u$ in $L^t \cap L^{t'}$ and so $S_{k,u_n} \rarr S_{k,u}$ is the Mikhlin-Itskovich algebra
$\calE_p$. It then follows from Theorem \ref{thm:GKK} that $D(k;u_n) \rarr D(k;u)$ uniformly in $k \in \bbC$. Choosing $N$ so that $\sup_{k \in \bbC, \, n \geq N} |D(k;u_n) - D(k;u)| < \delta$
gives the desired conclusion.
\end{proof}

Next, we study continuity of the map $u \mapsto m_2(\dotarg,\diamond;u)$. As always we fix $u \in X_\alpha$ for some $\alpha \in (1/2,1)$ and an admissible pair $(p,t)$. 
It follows from \eqref{CGO.scal} -- \eqref{CGO.scal.m2} that
\begin{equation}
\label{m2.sol}
m_2 = e_{-k} V_{k,u} 1 + e_{-k} V_{k,u} \left( \left(I-S_{k,u} \right)^{-1} S_{k,u} 1 \right)
\end{equation} 
so, to prove the continuity, it suffices to prove that 
\begin{itemize}
\item[(i)] $e_{-k} V_{k,u_n} 1  \rarr e_{-k} V_{k,u} 1$ in $L^\infty(\Omega_\delta;L^p(\bbC))$ as 
$n \rarr \infty$,
\item[(ii)] $S_{k,u_n} 1 \rarr S_{k,u} 1$ in $L^\infty(\Omega_\delta,L^p(\bbC))$ as $n \rarr \infty$,
\item[(iii)] $(I-S_{k,u_n})^{-1} \rarr (I-S_{k,u})^{-1}$ in $\calB(L^p)$ as $n \rarr \infty$, uniformly $k \in \Omega_\delta$.
\end{itemize}
As before, we may always assume that $k \in \Omega_\delta$ belongs to $\bbC \setminus Z_n$ if $n$ is large enough. 

First, we show:

\begin{lemma}
\label{lemma:UV}
Fix $u \in X_\alpha$ for some $\alpha \in (1/2,1)$, and let $(p,t)$ be an admissible pair.  Let $\{ u_n \}$ be a sequence from $X_\alpha$ converging to $u$. 
\begin{itemize}
\item[(i)]		$e_{-k} V_{k,u_n} 1  \rarr e_{-k} V_{k,u} 1$ in $L^\infty(\bbC;L^p(\bbC))$, and
\item[(ii)]	$S_{k,u_n} 1 \rarr S_{k,u} 1$ in $L^\infty(\bbC;L^p(\bbC))$ as $n \rarr \infty$,
\end{itemize}
\end{lemma}

\begin{proof}
Since $u_n \rarr u$ in $X_\alpha$, we also have $u_n \rarr u$ in $L^{t} \cap L^{t'}$, and hence in $L^2$. From \eqref{P1} and H\"{o}lder's inequality we have 
\begin{equation}
\label{Cauchy.2p}
\norm{\Cauchy\left[ uf \right]}{p} \lesssim_{\, p} \norm{u}{2} \norm{f}{p}.
\end{equation}
The conclusions (i) and (ii) follow from this estimate.
\end{proof}

The resolvents $R_n = (I-S_{k,u_n})^{-1}$ and $R=(I-S_{k,u})^{-1}$ exist for $k \in Z_\delta$ by Lemma 
\ref{lemma:Z}. Observe that
$$
R_n - R = R_n \left( W_{k,u_n} \circ V_{k,u_n} - W_{k,u} \circ V_{k,u} \right) R 
$$
so that continuity of the resolvent will follow from (i) estimates on $\norm{R}{\calB(L^p)}$ and $\norm{R_n}{L^p}$ uniform in $k \in \Omega_\delta$, (ii) uniform estimates on $\norm{W_{k,u_n}}{\calB(L^p)}$, $\norm{V_{k,u_n}}{\calB(L^p)}$, and (iii) norm estimates on $\norm{W_{k,u_n} - W_{k,u}}{\calB(L^p)}$
and $\norm{V_{k,u_n} - V_{k,u}}{\calB(L^p)}$ which vanish as $n \rarr \infty$. The uniform estimates (ii)
and the norm estimates (iii) follow from \eqref{Cauchy.2p}. Thus, it remains to prove uniform estimates on 
the resolvents $R$ and $R_n$. For this, the following estimate will suffice.

\begin{lemma}
\label{lemma:R.bd}
Suppose that $u \in X_\alpha$, that $(t,p)$ is an admissible pair, and $\delta>0$. Then
$ \sup_{k \in \Omega_\delta} \left(I - S_{k,u} \right)^{-1} \lesssim_{\, \delta} 1$
\end{lemma}

\begin{proof}
We will show first that $(I-S_{k,u})^{-1}$ has norm bounded by $2$ for all $k \in \bbC$ with $|k| \geq R$
for some constant $R$ depending on $\norm{u}{X_\alpha}$. We will then use a continuity-compactness argument to show that 
\begin{equation}
\label{con.com}
 \sup_{k \in \bbC: |k| \leq R, \dist(k,Z) \geq \delta/2} \norm{\left(I-S_{k,u}\right)^{-1}}{\calB(L^p)} \lesssim 1
\end{equation}
where the implied constant depends on $\delta$, $R$, and $u$.

First, we recall from \cite[Equation (3.13)]{Perry:2016} the
estimate
\begin{equation}
\label{pre.Tk.decay}
 \norm{T_k \psi}{p} 
 		\lesssim_{\, p} 	\left\langle k \right\rangle^{-1}
 						\left( 
 							\norm{u}{2} \norm{\psi}{p} + \norm{\dbar u}{2} \norm{\psi}{p} + 
 							\norm{u}{p} \norm {\dbar \psibar}{2} 
 						\right)
\end{equation}
Putting $\psi  = V_{k,u} h$ we recover
$$
\norm{S_{k,u} h}{p} \lesssim_{\, p}
	\langle k \rangle^{-1} \left( \norm{u}{2}^2 + \norm{u}{2} \norm{\dbar u}{2} + \norm{u}{p} \norm{u}{2p/(p-2)} \right) \norm{h}{p}.
$$
(recall that, for an admissible pair, $\norm{u}{2p/(p+2)}$ is bounded by
$\norm{u}{L^{t} \cap L^{t'}}$).
This shows that $\norm{S_{k,u}}{\calB(L^p)} < 1/2$ for $k$ sufficiently large
depending on $\norm{u}{X_\alpha}$.

It remains to prove \eqref{con.com}. The set 
$$U_{\delta,R} = \{k \in \bbC: |k| \leq R, \, \dist(k,Z) \geq \delta/2\}$$ 
is a compact subset of $\bbC$, while the map
$ k \rarr (I-S_{k,u})^{-1} $
is continuous from $U_{\delta,R}$ into $\calB(L^p)$. It follows that
$$ \sup_{k \in U_{\delta,R}} (I-S_{k,u})^{-1} \lesssim_{\, \delta,R} 1 $$
\end{proof}

\begin{proof}[Proof of Theorem \ref{thm:dbar}]
An immediate consequence of Propositions \ref{prop:D.dbar} and \ref{prop:todo}
together with Lemmas \ref{lemma:Z}, \ref{lemma:UV}, \and \ref{lemma:R.bd}.
\end{proof}


\section{The One-Soliton Solution}
\label{sec:one-soliton}

We now consider the one-soliton potential \cite{APP:1989}
\begin{equation}
\label{u.1s}
u_0(z) = \frac{2 e_{k_0}}{\rho(z)^2}
\end{equation}
(recall \eqref{rho}).  With this choice of $u$, \eqref{CGO.m.sym} admits the formal solution
\begin{subequations}
\label{m.1s}
\begin{align}
m_1(z,k) &= 1+\frac{1}{k-k_0} \frac{i}{1+|z|^2} \zbar\\
m_2(z,k) &= \frac{1}{k - {k_0}} \frac{i}{1+|z|^2} e_{-k_0}(z)
\end{align}
\end{subequations}
Using \eqref{data.s.lim}-\eqref{data.r.lim}, we read off
\begin{align}
\label{s.1s}
s(k) &= \frac{2i}{k-k_0}\\
\label{r.1s}
r(k) &= 0
\end{align}
The formal solution \eqref{m.1s} is correct for large $|k|$ since equation \eqref{CGO.m} has a unique solution for $|k|$ sufficiently large. To conclude that this equation holds for \emph{all} $k\neq k_0$ we must show that $k_0$ is the only exceptional point.  

In what follows, we will set $\kappa= k-k_0$ and define 
\begin{equation}
\label{T.soliton}
T(\kappa)=S_{k_0+\kappa,u_0}= \Cauchy e_{-\kappa} \rho^{-2} \CauchyBar e_\kappa \rho^{-2}.
\end{equation} 

We will prove:

\begin{theorem}
\label{thm:1ss}
For $u_0$ given by \eqref{u.1s}, the operator $I-T(\kappa)$ has a nontrivial nullspace if and only if $\kappa=0$.  Moreover, the zero eigenvalue of $I-T(0)$ is semisimple and of multiplicity two, and the zero eigenvalue of $(I-G(k_0,u_0))$ is also semisimple of multiplicity two. Finally, if $P(0)$ projects onto the nullspace of $I-T(0)$, then
\begin{equation}
\label{T.Matrix}
P(0) T(\kappa) P(0) 
	= \bigtwomat{1}{i\kappabar}{-i\kappa}{1}+ \calO_{\delta}\left( |\kappa|^{2-\delta} \right)
\end{equation}
for any $\delta >0$.
\end{theorem}

\begin{remark} The error estimate in \eqref{T.Matrix} can be improved to $\bigO{|\kappa|^2 \log |\kappa|}$ but we will not need this. We will need \eqref{T.Matrix} for the perturbation calculations in \S \ref{sec:perturb}. 
\end{remark}

\begin{proof} The theorem is an immediate consequence of Propositions \ref{prop:diff}, \ref{prop:two}, and \ref{prop:pert} below.
\end{proof}

We will prove Theorem \ref{thm:1ss} in three steps. First, we show that $T(\kappa)$ is a differentiable operator-valued function in the $\calE_p$ norm (which is stronger than the operator norm on $L^p$). Next, we use the determinant $\Det(I-T(\kappa))$ to prove that there is a unique singular point, and compute the determinant explicitly. Combining this explicit formula together and the fact that $T(0)$ is conjugate to a self-adjoint operator, we show that the zero eigenvalue of $I-T(0)$ is semisimple and of multiplicity two. A similar argument shows that the zero eigenvalue of $I -G(k_0,u_0)$ has the same property.
Finally, we use perturbation theory to obtain the formula \eqref{T.Matrix}.

\subsection{Smooth Dependence on $\kappa$}

The operator $T(\kappa)$ defined in \eqref{T.soliton}  belongs to the algebra $\calE_p$ of integral operators on $L^p(\bbC)$ for any  $p>2$ and each $\kappa \in \bbC$.
The operator $T(0)$ has an eigenvalue $\lambda=1$ since the functions 
\begin{equation}
\label{1s.scalar.basis}
\psi_1(z) = \zbar \rho(z)^{-2}, \quad \psi_2(z) = \rho(z)^{-2}
\end{equation}
are eigenvectors by \eqref{PPbar.int}. We will show that the eigenvalue 
$\lambda=1$ is a semisimple eigenvalue with multiplicity two, so the functions \eqref{1s.scalar.basis} span the eigenspace. First, we note:
\begin{proposition}
\label{prop:diff}
The map $\kappa \rarr T(\kappa)$ is differentiable as a map from $\bbC$ to 
$\calE_p$ for any $p>2$.
\end{proposition}

\begin{proof}
To prove that the map is  differentiable, we need to show that the formally obvious derivatives with respect to $\kappa$ and $\kappabar$ exist in $\calE_p$. First note that the operator $T(\kappa)$ has integral kernel
$$ K(z,w;\kappa) 
	= -\frac{1}{\pi^2} 
			\int \frac{1}{z-z'} \, \rho(z')^{-2}
				\, \frac{e_\kappa(w-z')}{\zbar'-\wbar} \, \rho(w)^{-2} \, dm(z')
$$

To show that $K(z,w;\kappa)$ is differentiable in $\kappabar$, we need to show that the function
$$ 
W(z,w,\kappa,h)=\int \frac{1}{z-z'} \, \rho(z')^{-2} 
				\left( 
						\frac{G(z'-w,h)}
								{\zbar'-\wbar} 
				\right) e_\kappa(w-z')
				 \rho(w)^{-2}\, dm(z')
$$
is $o(|h|)$ in $\calE_p$, where, for $z=x_1+ix_2$ and $h=h_1+ih_2$,
$$ G(z,h)=\frac{1}{2}\left(e^{2ih_1x_1}+e^{-2ih_2x_2}\right)-1-(ih_1x_1-ih_2x_2).$$
converges to zero in $\calE_p$ as $h \rarr 0$. From the trivial estimate
$$ \left| \frac{G(z,h)}{\zbar} \right| \leq 2^{1-2\theta}{ |h|^{1+\theta} } |z|^\theta $$
and the inequality $|z'-w|^\theta \leq 2^\theta(|z'|^\theta+|w|^\theta$ we have
$$
\left| W(z,w,\kappa,h \right)| 
	\leq 2|h|^{1+\theta} \int \frac{1}{|z-z'|}
	\, \rho(z')^{-2} 
		\left(|z'|^\theta+|w|^\theta \right)  \rho(w)^{-2} \, dm(z')
$$
Fix $p>2$ and choose $\theta$ so that $|z|^\theta \rho(z)^{-2} \in L^{2p/(p+2)}\cap L^{p'}$. The right-hand side is a sum of $|h|^{1+\theta}$ times two terms of the form $(\Cauchy f)(z) g(w)$ where $\Cauchy f \in L^p$ and $g \in L^{p'}$. It is now immediate that $W(z,w,\kappa,h)$ is $o(|h|)$ in $\calE_p$-norm as $h \rarr 0$.

The proof that $T(\kappa)$ is differentiable with respect to $\kappa$ is similar and is omitted.
\end{proof}

Since $T(\kappa)$ is differentiable, it follows that $\Det(I-T(\kappa))$ is also differentiable and we may use the $\dbar_\kappa$ equation for the determinant to study the behavior of $\Det(I-T(\kappa))$, compute the dimension of $\ker(I-T(0))$, and study the splitting of eigenvalues for $\kappa \neq 0$. 

\subsection{Determinant, Eigenvalue Multiplicity}
In this subsection, we prove:

\begin{proposition} 
\label{prop:two}
The operator $(I-T(0))$ has a semi-simple eigenvalue of multiplicity $2$ at $\lambda=1$. 
\end{proposition}

\begin{remark} This proves ``Assumption 1'' in Gadyl'shin-Kiselev's analysis of the one-soliton solution (see the remarks in \cite[p.\ 6091]{Kiselev:2006}).
\end{remark}

The proof is in several steps.

First, we show that the space $\ker_{L^p}\left(I-T(0)\right)$ has dimension exactly two by computing the determinant $\Det(I-T(\kappa))$. Formally we can integrate the $\dbar$-equation \eqref{D.dbar} which, in our case, reads
\begin{equation}
\label{D.Dbar.1s}
\dbar_k \log D(k) = \frac{1}{\kbar-\overline{k_0}}-c(k).
\end{equation}
From Remark \ref{rem:c} we have
$$ 
c(k) = \frac{1}{4\pi} \int \frac{1}{\kbar-\zetabar}\left| g(\zeta)\right|^2 \, dm(\zeta) $$
where
$$ g(k) = \frac{1}{\pi} \int e_{-k}(w) u_0(w) \, dm(w).$$
From \eqref{u.1s} we have
$$ |g(k)|^2 = G(k-k_0) $$
where $G(\kappa)$ is a rapidly decreasing, radial function of $\kappa$ since $(1+|z|^2)^{-1}$ is a radial, smooth decaying function with integrable derivatives of all orders. Letting $c(k)=\gamma(k-k_0)$, it follows that $\gamma$ admits a large-$\kappa$ asymptotic  expansion of the form
$$ \gamma(\kappa) \sim \sum_{j=0}^\infty c_j \kappabar^{-j-1}. $$
where $c_j = (4\pi)^{-1} \int \zetabar^{j} G(\zeta) \, dm(\zeta)$. On the other hand $\gamma(\kappa)$ has a finite limit as $\kappa \rarr 0$. 
Moreover, since $G(\kappa)$ is radial, all of the $c_j$ with $j \geq 1$ vanish. By unitarity of the transform $\calF$, 
$$ c_0 = \frac{1}{4\pi} \int \left|u_0(z) \right|^2 \, dm(z) = 1. $$
If we now let $D(k)=\Delta(k-k_0)$, it follows that $\Delta(\kappa)$ obeys the
$\dbar$-problem
\begin{subequations}
\label{Delta.dbar}
\begin{align}
\dbar_k \log \Delta(\kappa) &= \frac{1}{\kappa}-\gamma(\kappa)\\
\lim_{|\kappa| \rarr \infty} \Delta(\kappa)&=1
\end{align}
\end{subequations}
and for any positive integer $N$, 
$$ 
\dbar_\kappa \log \Delta(\kappa) = 
\begin{cases}
\bigO{|\kappa|^{-N}} & |\kappa| \rarr \infty,\\[0.2cm]
\kappa^{-1} + \bigO{1} & \kappa \rarr 0,
\end{cases}
$$
\emph{presuming} that the expression \eqref{s.1s} remains correct. If this is so, we can integrate formally to find that 
\begin{equation}
\label{Delta.asy}
\log \Delta(\kappa)= \log |\kappa|^2 + \bigO{1} 
\end{equation}
as $\kappa \rarr 0$, and conclude that $\kappa=0$ is a zero of multiplicity two for $\Delta(\kappa)$. 

To prove that this is the case, we must know that the solution \eqref{m.1s} is correct for all $k \neq k_0$, which will be the case provided $\Delta(\kappa)=0$ for all nonzero $\kappa$. Observe that $\Delta(\kappa)$ is radial: if $\calU(\theta)$ is the isometry $\left(\calU(\theta)f \right)(\kappa)=f(e^{i\theta}\kappa)$ then
$$ \calU(\theta) T(\kappa) \calU(\theta)^{-1} = T(e^{i\theta}\kappa)$$
so that
$$ \Delta(\kappa)=\Delta(e^{i\theta}\kappa).$$
Now let $\alpha$ be the modulus of the first zero of $\Delta(\kappa)$. We claim that $\alpha=0$. Writing $\Delta(\kappa)=H(|\kappa|^2)$ for $H:(0,\infty) \rarr \bbC$, it follows from \eqref{Delta.dbar} and \eqref{Delta.asy} that
\begin{align*}
\frac{d}{dt}\log H(t) &= t^{-1}-h(t),\\
\lim_{t \rarr \infty}H(t) &=1
\end{align*}
where 
$$
h(t)=
\begin{cases}
t^{-1} + \bigO{t^{-N}} & t \uarr \infty,\\
\bigO{1} & t \darr 0.
\end{cases}
$$
If $\alpha \neq 0$, we can integrate from $\alpha$ to $\infty$ to obtain
$$ \log H(\alpha) = \log \alpha+\int_\alpha^1 h(t) \, dt-\int_1^{\infty}\left(t^{-1}-h(t)\right) \, dt $$
a contradiction since then $\log H(\alpha)$ is finite and hence $H(\alpha) \neq 0$. 
We conclude that $H(t)$ has no zeros in $(0,\infty)$, so $D(k)$ has no zeros for $|k-k_0|>0$.
We also have the  formula
$$ H(t)=ct \exp \left( -\int_0^t h(s) \, ds\right) $$
for $t \in (0,1)$, where
\begin{equation}
\label{H.c}
c = \exp\left(\int_0^1 h(t) \, dt - \int_1^\infty \left(t^{-1}-h(t)\right) \, dt \right).
\end{equation}

We have proved:

\begin{lemma}
\label{lemma:T.det}
The determinant $\Det(I-T(k,u_0))$ has no zeros for $k \neq k_0$, and 
$$ \Det(I-T(k,u_0))=c|k-k_0|^2 \left(1+\bigO{|k-k_0|^2}\right) $$
as $k \rarr k_0$, where $c$ is given by \eqref{H.c}.
\end{lemma}

We will now use this fact to show that the nullspace of $I-T(k_0,u_0)$ is two-dimensional. Let $T_0 = T(k_0,u_0)$ and recall \eqref{rho}.
 A short computation shows that 
$$ \rho^{-1} T_0 \rho = B^* B $$
where
$$ B = \rho^{-1} \Cauchy \left[  \rho^{-1} \left( \dotarg \right) \right]. $$
The operator $B^*B$ is positive and compact as an operator from $L^2(\bbC)$ to itself. We now apply Lemma \ref{lemma:multiplicities} to the family
$$ T^\sharp(\kappa) = \rho^{-1} T(k_0+\kappa,u_0) \rho $$
viewed as operators on $L^2$. Note that $T^\sharp(0)$ has a semisimple eigenvalue at $\lambda=1$. From Lemma \ref{lemma:T.det} we have
$$ \Det(I-T^\sharp(\kappa)) = c|\kappa|^2(1+\bigO{|k|^2}) $$
for a positive constant $c$ and hence, by Lemma \ref{lemma:multiplicities}, the kernel of $(I-T^\sharp(0))$ as at most two-dimensional. On the other hand, using the identities \eqref{PPbar.int}, it is easy to see that the orthonormal vectors \eqref{1s.scalar.basis}
 belong to $\ker(I-T^\sharp_0)$. Hence, we have proved:

\begin{lemma}
\label{lemma:2d}
The operator $I-T(0)$ has a two-dimensional kernel.
\end{lemma}

\begin{proof}[Proof of Proposition \ref{prop:two}]  An immediate consequence of Lemma \ref{lemma:2d}.
\end{proof}

Using \eqref{PPbar.int} it is easy to check that the orthonormal vectors 
\begin{equation}
\label{1s.basis}
\Psi_1 = \pi^{-1/2} \rho^{-3} \twovec{\zbar}{1}, \quad
\Psi_2 = \pi^{-1/2} \rho^{-3} \twovec{1}{-z}.
\end{equation} 
belong to $\ker_{L^2 \oplus L^2} (I-G_0^\sharp)$. Hence, these form a basis for $\ker_{L^2 \oplus L^2} (I-G_0^\sharp)$. 

\subsection{Eigenvalue Splitting}

In this subsection, we prove:

\begin{proposition}
\label{prop:pert}
The asymptotic formula \eqref{T.Matrix} holds for small $\kappa$. 
\end{proposition}

We begin by computing the Laurent expansion of $T(\kappa)$ about $\kappa=0$ and the splitting of the eigenvalue $\lambda=1$ at $\kappa=0$.  Denote by $T(\kappa)'$ the Banach space adjoint of $T(\kappa)$ with respect to the dual pairing
\begin{equation}
\label{dual}
\left\langle f, g \right\rangle = \int_{\bbC} f(z) g(z) \, dm(z)
\end{equation} 
of $L^{p'}$ and $L^p$. Then
\begin{equation}
\label{T.dual}
T(\kappa)'=-\rho^{-2}e_\kappa \CauchyBar e_{-\kappa} \rho^{-2} \Cauchy. 
\end{equation}
Using \eqref{PPbar.int},  it is not difficult to see that the $\lambda=1$ eigenspace of $T(0)'$ is spanned by the vectors 
\begin{equation}
\label{1s.basis.dual}
\chi_1(z)={\frac{2}{\pi}}\, z\rho(z)^{-4}, \qquad  \chi_2(z)= {\frac{2}{\pi}}\,\rho(z)^{-4}.
\end{equation}
It is easy to check that $\left\langle \chi_i , \psi_j \right\rangle = \delta_{ij}$. It now follows that for $\kappa$ and $\lambda-1$ small,
$$ (I-\lambda T(\kappa))^{-1} = \frac{F}{\lambda-1} + \bigO{1} $$
as bounded operators on $L^p$, where
\begin{equation}
\label{F}
F 
	= \left\langle \chi_1, \cdot \right\rangle \psi_1 +
		\left\langle \chi_2, \cdot \right\rangle \psi_2.
\end{equation}
To compute the splitting of the eigenvalue $\lambda=1$ for $\kappa$ small and nonzero, we first note that there is an $r>0$ with the property that $\| (\lambda I - T(\kappa))^{-1} \|$ is bounded for all $\lambda$ on the circle $|\lambda -1|=r$ and all $\kappa$ sufficiently small. The projection 
$$ P(\kappa) = \frac{1}{2\pi i} \oint_{\Gamma_r} (\lambda I - T(\kappa))^{-1} \, d\lambda $$
has rank two for $\kappa$ small. Moreover, since $T(\kappa)$ is differentiable as a $\calB(L^p)$ operator-valued function, it follows that $P(\kappa)$ is also differentiable as an operator-valued function. We wish to compute the eigenvalues of the rank-two operator $P(\kappa) T(\kappa) P(\kappa)$ using ideas of \S \ref{subsec:Kato-Rellich}. 

Let $W(\kappa)=P(\kappa)-P(0)$. Since $P(\kappa)$ is differentiable it follows that 
$$\norm{W(\kappa)}{ } =\bigO{|\kappa|} \text{ as }\kappa \rarr 0.$$  
Now let
$$ V(\kappa) = (I - W(\kappa))^{-1/2}\left[ I + P(\kappa)W(\kappa) + W(\kappa)P(0) \right]$$
(compare \eqref{V}).  As an operator on $\calB(L^p)$, 
\begin{equation}
\label{V.kappa}
 V(\kappa) = I + P(0)W(\kappa)+W(\kappa)P(0) + \bigO{|\kappa|^2 }
 \end{equation}
so that
\begin{equation}
\label{V.kappa.inverse}
V(\kappa)^{-1} = I - P(0) W(\kappa) - W(\kappa) P(0) + \bigO{|\kappa|^2}. 
\end{equation}
We will now compute the eigenvalues of $P(\kappa) T(\kappa) P(\kappa)$ by computing those of the operator $P(0) V(\kappa)^{-1} T(\kappa) V(\kappa) P(0)$ (see \eqref{A.similar}). Since $P(0)T(0) P(0)$ is diagonal, the commutators of $P(0) T(0) P(0)$ with $P(0) W(\kappa) P(0)$ vanish. Using this fact, the differentiation of $T(\kappa)$, and the asymptotic formulas \eqref{V.kappa}-\eqref{V.kappa.inverse}, it is not difficult to see that
$$ 
P(0) V(\kappa)^{-1} T(\kappa) V(\kappa) P(0) 
	=  P(0) T(\kappa) P(0) +\bigO{|\kappa|^2}.
$$
We now compute the matrix of $P(0)T(\kappa) P(0)$ using
$\{ \psi_1, \psi_2 \}$ and $\{\chi_1, \chi_2\}$ as respective basis sets for the domain and range. This entails evaluating the integrals
\begin{subequations}
\label{M.int}
\begin{align}
M_{11}(\kappa) 
	&= - \frac{2}{\pi^3} 
	\int z \rho(z)^{-4}
			 \frac{1}{z-z'} \, \rho(z')^{-2}
				\, \frac{e_\kappa(w-z')}{\zbar'-\wbar} \, 
				\wbar \rho(w)^{-4}
			 \, dm(w,z',z), \\[0.2cm]
M_{12}(\kappa) 
	&= - \frac{2}{\pi^3} 
	\int  z \rho(z)^{-4}
			\frac{1}{z-z'} \, \rho(z')^{-2}
				\, \frac{e_\kappa(w-z')}{\zbar'-\wbar} \, 
				\rho(w)^{-4}
			 \, dm(w,z',z), \\[0.2cm]
M_{21}(\kappa) 
	&= - \frac{2}{\pi^3} 
	\int  \rho(z)^{-4}
			 \frac{1}{z-z'} \, \rho(z')^{-2}
				\, \frac{e_\kappa(w-z')}{\zbar'-\wbar} \,
			\wbar \rho(w)^{-4}
			\, dm(w,z',z), \\[0.2cm]
M_{22}(\kappa) 
	&= - \frac{2}{\pi^3} 
	\int  \rho(z)^{-4}
			\frac{1}{z-z'} \, \rho(z')^{-2}
				\, \frac{e_\kappa(w-z')}{\zbar'-\wbar} \, 
				\rho(w)^{-4}
			 \, dm(w,z',z). 
\end{align}
\end{subequations}
We will give hints to evaluate $M_{11}(\kappa)$; the others are similar.
Since the integral is absolutely convergent we may carry out the $z$-integration first using \eqref{PPbar.int}. The result is
$$
M_{11}(\kappa) 
		= \frac{2}{\pi^2}
			\int 
			 	 \rho(z')^{-4} \, \wbar \rho(w)^{-4}
			 		\frac{e_\kappa(w-z')}{\zbar'-\wbar} 
			 			 \, dm(w,z')
$$
Using the estimate
$$ \left|e_\kappa(w) -1 -i\kappa w \right| \leq C_\delta |w|^{2-\delta} |\kappa|^{2-\delta} $$
we see that
$$
M_{11}(\kappa) = a+b\kappa +c \kappabar + \calO_\delta\left(|\kappa|^{2-\delta}\right)
$$
where $a=1$ and $b=c=0$ by direct computation, using \eqref{PPbar.int}. Similar calculations for the remaining integrals show that
\eqref{T.Matrix} holds.

\begin{proof}[Proof of Proposition \ref{prop:pert}] An immediate consequence of the computations above.
\end{proof}


\section{Perturbation of the One-Soliton Solution}
\label{sec:perturb}

We now show that for  $\varphi \in C_0^\infty(\bbC)$ satisfying a Fourier transform condition, and $\eps$ small, $u_0 + \eps \varphi$ has no soliton. This result is originally due to Gadyl'shin and Kiselev \cite{GK:1996,GK:1999} although we achieve some simplification of the proof and remove their Assumption 1.

We  consider perturbations of the form $u=u_0 +\eps \varphi$ for $\varphi \in C_0^\infty(\bbC)$. For computational convenience, we will set $\chi= \varphi e_{k_0}$.

To study the perturbations, we study the spectrum   of the operator
\begin{align}
\label{T.eps}
T(\kappa,\eps) 
&= S_{k_0+\kappa,u_0+\eps \varphi}	\\
\nonumber
&=\Cauchy e_{-\kappa} \rho^{-2} \CauchyBar \rho^{-2} e_\kappa\\
\nonumber
	&\quad
		+ \eps 
			\left( 
				\Cauchy e_{-\kappa} \rho^{-2} \CauchyBar e_\kappa \chibar
	 			+ \Cauchy e_{-\kappa} \chi \CauchyBar e_\kappa \rho^{-2}
	 		\right)
\\
\nonumber
&\quad
	+\eps^2 \, \Cauchy e_{-\kappa} \chi \CauchyBar e_\kappa \chibar
\end{align}
Note that
$$ 
T(\kappa,\eps) = T(\kappa,0)-T(0,0) + \bigO{\eps}
$$
in the $\calB(L^p)$ operator norm, and note that $T(\kappa,0)$ is the operator
$T(\kappa)$ from the preceding section. Let us denote by $P$ and $P_0$ the respective projections
$$ 
P(\kappa,\eps) = \frac{1}{2\pi i}
	\oint_{\Gamma_r} \left(\lambda I - T(k,\eps) \right)^{-1} \, d\lambda
$$
and
$$
P(0,0)= \frac{1}{2\pi i}
	\oint_{\Gamma_r} \left( \lambda I - T(0,0) \right)^{-1} \, d\lambda.
$$
It is easy to see that, as operators from $L^p$ to itself,
$$ \| P - P_0 \| = \bigO{\eps+ |\kappa|}$$
since $\kappa \mapsto T(\kappa,0)$ is differentiable at $\kappa=0$. In what follows, we will write $T$ for $T(\kappa,\eps)$ and $T_0$ for $T(0,0)$. 

We will prove:
\begin{theorem}
\label{thm:split}
Let $u_0$ be the one-soliton solution \eqref{u.1s} and let $\varphi \in C_0^\infty(\bbC)$. 
\begin{itemize}
\item[(i)] For $\eps$ and $\kappa = k-k_0$ small, the asymptotic formula
$$ \det(I-T(\kappa,\eps)) =   |i\kappabar + \eps \beta|^2 + \eps^2 |\alpha|^2 
	+ o\left(\eps^2+\eps|\kappa| + |\kappa|^2\right)$$
holds, 
where, setting $\chi=e_{k_0} \varphi$,
\begin{align*}
\alpha 
	&= 	-\frac{2}{\pi}\int \left(\chi  -\chibar |z|^2 \right) \rho^{-4} \, dm(z)\\
\beta
	&=	\frac{2}{\pi}\int (\chi - \chibar) z \rho^{-4} \, dm(z)
\end{align*}
In particular, let $C>0$ be given. If $\eps \neq 0$ is sufficiently small and $\alpha \neq 0$, then $I-T(\kappa,\eps)$ has trivial kernel for $|\kappa| < C\eps$.
\item[(ii)] There is a $C>0$ so that $(I-T(\kappa,\eps))$ is invertible for all sufficiently small $\eps>0$ and $\kappa$ with $\kappa>C \eps$.
\end{itemize}
\end{theorem}

\begin{proof}
(i) We wish to show that the rank-two operator
$ P \left( T(\kappa,\eps) - T(0,0) \right) P $
has nonzero eigenvalues for all small $\kappa$, $\eps$. By Lemma \ref{lemma:APQW}, this operator is similar to the rank-two operator
$$ P_0 V^{-1} \left(T-T_0 \right) V P_0 $$
where
$$ V = (I-(P-P_0)^2)^{-1/2} \left(I - (P-P_0)^2 +[P,P_0] \right). $$
We will show that $\det (P_0 V^{-1} \left(T-T_0\right) V P_0$ is nonvanishing for all sufficiently small $\eps$, $\kappa$.

First, we note some reductions. Let $\delta T= T-T_0$. Note that
$$ \| \delta T \| = \bigO{\eps + |\kappa|}, \quad \| P - P_0 \| = \bigO{\eps +|\kappa|}.$$
Then
\begin{align*}
P_0 V^{-1} \delta T V P_0
	&=	P_0 	\left(
						I +\left[P,P_0 \right] 
					\right) \delta T 
					\left(
						I-\left[P,P_0\right] 
					\right) 
			P_0\\
	&=	A + B
\end{align*}
modulo terms of order $o\left(\left(\eps+ |\kappa|\right)^2\right)$, where
$$ A=P_0 \, \delta T \, P_0, \qquad B  = P_0 \left[ \, \left[ P, P_0 \right], \, \delta T \right]P_0.$$ From the identity
$$ 
\det(A+B)
	= \det A + \det B+ 
			\left| \begin{array}{ll} a_{11} & a_{12} \\ b_{21} & b_{22} \end{array} \right| +
			\left| \begin{array}{ll} b_{11} & b_{12} \\ a_{21} & a_{22} \end{array} \right|
$$
the estimate $\| A \| = \bigO{\eps + |\kappa|}$, and the estimate $\| B \| = \bigO{\eps^2 + \eps |\kappa| + |\kappa|^2}$, it follows that
\begin{equation}
\label{pert.det}
 \det(P_0 V^{-1} \delta T V P_0) = \det(A) + o(\eps^2+ \eps |\kappa|+|\kappa|^2) 
\end{equation}

We will now calculate $\det A$ using the fact that $P_0 = F $ (see \eqref{F}). Note that
\begin{equation}
\label{delta.T}
\delta T 
	= \left[T(\kappa,0)-T(0,0)\right] + \eps
		 T^{(1)}(\kappa) 	+ \bigO{\eps^2 + \eps |\kappa|}.
\end{equation}
where
$$ 
T^{(1)}(\kappa) = 
		\left( 
			\Cauchy e_{-\kappa} \chi \CauchyBar e_\kappa \rho^{-2} + 
			\Cauchy e_{-\kappa} \rho^{-2} \CauchyBar e_\kappa \chibar 
		\right).
$$ 
We have already computed the matrix of $P_0 T(\kappa,0) P_0$ (see \eqref{T.Matrix}), 
while $P_0 T(0,0) P_0$ is the identity matrix. Hence
\begin{equation}
\label{T.delta.0}
 P_0 \left[T(\kappa,0)-T(0,0) \right] P_0 = \bigtwomat{0}{-i\kappabar}{i\kappa}{0} + \bigO{|\kappa|^{2-\delta}}.
\end{equation}
by \eqref{T.Matrix}. 

Next, observe that
$$ \eps P_0 T^{(1)}(\kappa) P_0 = \eps P_0 T^{(1)}(0) P_0 + \bigO{\eps |\kappa|}.$$
and
$$ 
T^{(1)}(0) = 
		\left( 
			\Cauchy  \chi \CauchyBar  \rho^{-2} + 
			\Cauchy  \rho^{-2} \CauchyBar  \chibar 
		\right).
$$
We may compute the matrix of $P_0 T^{(1)}(0) P_0$ with respect to the basis $\{ \psi_1, \psi_2 \}$ as
\begin{equation}
\label{T.pert}
M^{(1)} = \bigtwomat
	{\angles{\chi_1}{T^{(1)}(0)\psi_1}}
	{\angles{\chi_1}{T^{(1)}(0)\psi_2}}
	{\angles{\chi_2}{T^{(1)}(0)\psi_1}}
	{\angles{\chi_2}{T^{(1)}(0)\psi_2}}
\end{equation}
To carry out this computation, observe first that for any functions $f_1$ and $f_2$, 
$$
\angles{\chi_i}{\Cauchy f_1 
		  \CauchyBar 
		\overline{f_2} \psi_j}=
-\angles{\Cauchy \chi_i}{ f_1 \CauchyBar  \overline{ f_2} \psi_j}
$$ 
while, by \eqref{PPbar.int},
$$
\left( \Cauchy \chi_1\right)(z)  = - \frac{2}{\pi} \rho(z)^{-2}, \quad 
\left( \Cauchy \chi_2\right)(z)  = \frac{2}{\pi}\zbar \rho(z)^{-2}.
$$
Using these identities, and using \eqref{PPbar.int} to help compute the integrals, we find that 
\begin{equation}
\label{T.pert.matrix}
M^{(1)}=\bigtwomat{\alpha}{\beta}{-\betabar}{\alphabar}
\end{equation}
where
\begin{align*}
\alpha 
	&= 	-\frac{2}{\pi}\int \left(\chi  -\chibar |z|^2 \right) \rho^{-4} \, dm(z)\\
\beta
	&=	\frac{2}{\pi}\int (\chi - \chibar) z \rho^{-4} \, dm(z)
\end{align*}
It is natural to impose the orthogonality condition
$$ \int \chi \rho^{-2} \ dm(z) = 0 $$
which insures that the perturbation $\psi$ is orthogonal to the soliton solution $u_0$. 
In this case
\begin{align*}
\alpha
	&=	-\frac{2}{\pi} \int \left(\chi - \chibar\right)\rho^{-4} \, dm(z)\\
\beta
	&=	\frac{2}{\pi} \int \left(\chi - \chibar\right) z \rho^{-4} \, dm(z)
\end{align*}
Combining \eqref{T.Matrix}, \eqref{delta.T},  \eqref{T.delta.0} and \eqref{T.pert.matrix}, we conclude that
$$ A = \bigtwomat{\eps \alpha}{\eps \beta + i \kappabar}{-\eps \betabar - i \kappa}{\eps \alphabar} + o(\eps + |\kappa|)$$
so that
\begin{equation}
\label{det.A}
\det (A) = |i\kappabar + \eps \beta|^2 + \eps^2 |\alpha|^2 
	+ o\left(\eps^2+\eps|\kappa| + |\kappa|^2\right)
\end{equation}
It now follows from \eqref{det.A} and \eqref{pert.det} that
$$ \det(P_0 V^{-1} \delta T V P_0) =  |i\kappabar + \eps \beta|^2 + \eps^2 |\alpha|^2 
	+ o\left(\eps^2+\eps|\kappa| + |\kappa|^2\right)
$$\
Hence, if at least one of $\alpha$ and $\beta$ is nonzero, then the determinant is nonzero for \emph{all} sufficiently small $\eps$ and $|\kappa|$, including $\kappa=0$, so that $I-T(\kappa,\eps)$ has trivial kernel for such $\eps$ and $\kappa$.

(ii) This is a simple perturbation argument. In what follows, $\| \dotarg \|$ denotes the $\calB(L^p)$ operator norm. It follows from Theorem \ref{thm:1ss} and \eqref{T.Matrix} that 
$$ \| (I- T(\kappa))^{-1} \| \leq C_1 |\kappa|^{-1}$$
for a constant $C_1$ independent of $\kappa$. 
From this estimate and the second resolvent identity it is easy to see that
$$\left[ |\kappa| \| (I-T(\kappa,\eps)^{-1}\| \right]\leq C_1 + C_1 C_2 \eps |\kappa|^{-1} \left[ |\kappa|\|(I-T(\kappa,\eps))^{-1}\| \right] $$
where $C_2$ bounds $\eps^{-1} \left( T(\kappa,\eps) - T(\kappa) \right)$. It follows that for $|\kappa| \geq 2C_1 C_2 \eps$, the estimate
$$ |\kappa| \| (I-T(\kappa,\eps))^{-1} \| \leq 2 C_1 $$
holds, which shows that $(I-T(\kappa,\eps))$ is invertible for $\eps$ sufficiently small and all $\kappa$ with $\kappa \geq 2C_1 C_2 \eps$. 

\end{proof}


\begin{proof}[Proof of Theorem \ref{thm:main}]
First, using Theorem \ref{thm:split}(ii), pick $C_1>0$ and $\eps_0$ so that $(I-T(\kappa,\eps))$ is invertible for all $\eps<\eps_0$ and all $\kappa$ with $\kappa>C_1 \eps$. Next, by decreasing $\eps$ if needed, use Theorem \ref{thm:split}(i) with $C=2C_1$ to conclude that $(I-T(\kappa,\eps))$ is also invertible for $\kappa$ with
$|\kappa| <2C_1 \eps$. We now conclude that $(I-T(\kappa,\eps))$ is invertible for every $\kappa \in \bbC$ and all sufficiently small $\eps$, so that the exceptional set is empty.
\end{proof}


\appendix

\section{Renormalized Determinants}
\label{sec:GGK}

In this subsection we recall results of Gohberg, Goldberg, and Krupp (see their paper \cite{GGK:1997} and the monograph \cite{GGK:2000}) which will allow us to define a Hilbert-Carlene determinant for certain integral operators on $L^p(\bbC)$. 

We begin by recalling that, if $F$ is a finite-rank operator acting on a Banach space $X$,  
$$ \det(I+F)=\prod_{j}(1+\lambda_j(F)) $$
where $\{ \lambda_j(F) \}$ are the finitely many eigenvalues of $F$. This determinant is multiplicative, i.e., $\det((I+A)(I+B))=\det(I+A) \det(I+B)$, and obeys the identity
$$ \log \det(I+F) = \Tr\log(I+F) $$
when $F$ has small norm, where
$$ \log (I+F) = \sum_{n=1}^\infty \frac{(-1)^{n+1}}{n} \Tr(F^n). $$
A related, modified determinant is
$$ \Det(I+F) = \det\left((I+F) e^{-F}\right) $$
where $e^F$ is defined by Taylor's series for the exponential function.
Under certain circumstances, both $\det(I + \dotarg)$ and $\Det(I+\dotarg)$ can be extended to larger classes of compact operators acting on $X$. For example, if $X$ is a Hilbert space $\calH$, $\det(I + \dotarg)$ extends to the trace-class operators on $\calH$, and $\Det(I+\dotarg)$ extends to the Hilbert-Schmidt operators on $\calH$ (see, for example \cite[chapters 3 and 9]{Simon:2005}  or \cite[chapter 4]{GK:1969}.

Next, we recall the Mikhlin-Itskovich algebra of integral operators on $L^p(M,\mu)$ for a measure space $(M,\mu)$, following \cite[\S 5]{GGK:1997} (see also the monograph \cite{GGK:2000} for a detailed exposition). 
Let $p \in (1,\infty)$, $p^{-1}+q^{-1} = 1$, and denote by $L^{p,q}(M \times M)$ the Banach space of measurable functions $a: M \times M \rarr \bbC$ with the norm
$$ \| a \|_{p,q} = \left( \int_M \left( \int_M |a(x,y)|^q \, d\mu(y) \right)^{p/q} \, d\mu(x) \right)^{1/p}. $$ 
\begin{definition}
\label{def:Ep}
We denote by $\calE_p$ the linear space of integral operators 
$$ (Af)(x) = \int_M a(x,y) f(y) \, d\mu(y) $$
with $a \in L^{p,q}(M\times M)$ and $a^* \in L^{q,p}(M \times M)$, where
$$ a^*(x,y) = \overline{a(y,x)}. $$
We norm $\calE_p$ by
$$ \| A \|_{\calE_p} = \max\left( \| a \|_{p,q}, \| a^* \|_{q,p} \right). $$
\end{definition}

In \cite{GGK:1997}, it is shown that $\calE_p$ is an embedded subalgebra of the bounded linear operators on $L^p(M, d\mu)$, that  $$ \| AB \|_{\calE_p} \leq \| A \|_{\calE_p} \|B\|_{\calE_p},$$ and that finite-rank operators $F_{\calE_p}$ are norm-dense in $\calE_p$. 
Gohberg, Goldberg, and Krein prove:

\begin{theorem}[{\cite[\S 5]{GGK:1997}}] $~~$ \newline
\label{thm:GKK}
\begin{itemize}
\item[(i)] The trace maps $F \mapsto \Tr(F^n)$ have continuous extensions from $F_{\calE_p}$ to $\calE_p$ for every $n \geq 2$.
\item[(ii)] The determinant $\Det(I+F)$ has a continuous extension to $\calE_p$.
\item[(iii)]
For $F \in \calF$, we have
\begin{equation}
\label{Det.finite}
\Det(I + F) = \det(I+F) \exp(-\Tr(F)) 
\end{equation}
where  $\det\left( I + (\dotarg) \right)$ is the usual trace-class determinant. 
\end{itemize}
\end{theorem}

Note that when $p=2$, the Mikhlin-Itskovich algebra consists of the Hilbert-Schmidt operators with the usual norm, and the determinant $\Det\left( I + (\dotarg) \right)$ is the renormalized determinant $\det_2(I+(\dotarg))$ (see for example \cite{GK:1969,Simon:2005}).

\begin{remark}
\label{rem:Det.common}
Observe that the finite-rank operators $\calF_{\calE_p}$ take the form 
$\sum_{i=1}^n \langle \psi_i , \cdot \rangle \varphi_i$ where $\psi_i \in L^{p'}(M,\mu)$
and $\varphi_i \in L^p(M,\mu)$. Supposing that $(M,\mu)$ is a $\sigma$-finite measure space, the set $\calD$ of finite linear combinations of characteristic functions for sets of finite measure is dense in each $L^p(M,\mu)$. The set of finite-rank operators
with integral kernels of the form $\sum_{i=1}^n \psi_i(x) \varphi_i(y)$ for $\psi_i, \varphi_i \in \calD$ is therefore dense in $\calE_p$ for any $p$. This implies that if $A \in \calE_p \cap \calE_{p'}$, the determinants $\Det(I+A)$ defined on $\calE_p$ and $\calE_{p'}$ coincide.
\end{remark}

Using \eqref{Det.finite} and the multiplicative property of the ordinary determinant, we may easily show that
\begin{equation}
\label{Det.Multiply}
\Det\left[(I-B)(I-C)\right] = \Det(I-B)\Det(I-C)\exp(-\Tr(BC))
\end{equation}

The following variant of the standard formula for differentiation of determinants 
is used to derive the $\dbar$-equation \eqref{D.dbar}.

\begin{lemma}
\label{lemma:Det.dot}
Suppose that $t\mapsto A(t)$ is a differentiable
map from $\left(  -\varepsilon,\varepsilon\right)  $ into $\calE _{p}$
with the property that $t\mapsto A^{\prime}(t)$ is a continuous finite-rank
operator-valued function. Then
\[
\frac{d}{dt}\log\Det\left(  I+A(t)\right)  =\Tr%
\left(  \left(  I+A(t)\right)  ^{-1}A^{\prime}(t)\right)  -\Tr%
\left(  A^{\prime}(t)\right)
\]

\end{lemma}

\begin{proof}
First, if $t\mapsto F(t)$ is a differentiable family of finite-rank operators,
we have%
\begin{equation}
\frac{d}{dt}\log\det\left(  I+F(t)\right)  =\Tr\left(  \left(
I+F(t)\right)  ^{-1}F^{\prime}(t)\right)  \label{eq:logdetF.dot}.
\end{equation}

Now consider the operator $A(t)$ and its determinant. Writing
\[
A(t)=A(0)+\int_{0}^{t}A^{\prime}(s)~ds
\]
we can decompose $A(t)$ into a fixed operator $B=A(0)$ and a finite-rank
operator-valued function $F(t)=\int_{0}^{t}A^{\prime}(s)~ds$ pf small norm for
$\left\vert t\right\vert $ small. Suppose first that $\left(  I+B\right)  $ is
invertible. Using (\ref{Det.Multiply}) we compute%
\[
\Det\left(  I+A(t)\right)  =\Det\left(
I+B\right)  \det\left(  I+\left(  I+B\right)  ^{-1}F(t)\right)  \exp\left(
-\Tr\left(  F(t)\right)  \right)  .
\]
Differentiating and using (\ref{eq:logdetF.dot}) we have%
\begin{align*}
\frac{d}{dt}\log\Det\left(  I+A(t)\right)   &
=\Tr\left(  \left[  I+(I+B)^{-1}F(t)\right]  ^{-1}\left(
I+B\right)  ^{-1}F^{\prime}(t)\right)  -\Tr\left(  F^{\prime
}(t)\right) \\
&  =\Tr\left(  \left[  I+B+F(t)\right]  ^{-1}F^{\prime
}(t)\right)  -\Tr\left(  F^{\prime}(t)\right) \\
&  =\Tr\left(  \left[  I+A(t)\right]  ^{-1}A^{\prime
}(t)\right)  -\Tr\left(  A^{\prime}(t)\right)
\end{align*}
as was to be proved.

Now consider the case where $\left(  I+B\right)  $ is not invertible. Since
$B$ is compact, $\left(  I+zB\right)  $ has isolated singularities and so, for
some $\varepsilon\neq0$, $\left(  I+\left(  1+\varepsilon\right)  B\right)  $
is invertible. Write $\left(  I+A(t)\right)  =I+\left(  1+\varepsilon\right)
B+\left(  F-\varepsilon B\right)  $ and further decompose $F-\varepsilon
B=G+C$ where $G$ is finite rank and $C$ has small enough norm that $\left(
1+\varepsilon B+C\right)  $ is invertible. We then replace $B$ by $\left(
1+\varepsilon\right)  B+C$ and $F$ by $G$ and repeat the argument.
\end{proof}


\section{Estimates on an Integral Operator for the Direct Scattering Problem by Russell Brown}
\label{app:brown}

Recall that the operator $S_{k,u}$ is a compact operator on $L^{p}$ for any
$p>2$ provided $u\in L^{p}\cap L^{p^{\prime}}$. We will show that, for
suitable $p$, $S_{k,u}$ belongs to the Mikhlin-Itskovich algebra
$\calE _{p}$ { of integral operators on $L^p(\bbC)$ (Definition \ref{def:Ep})}, so that, { by Theorem \ref{thm:GKK}}, we may define a determinant 
$\Det\left(  I- S_{k,u}\right)  $ whose zeros are the points of the exceptional set. We will prove:

\begin{proposition}
\label{prop:brown}
Suppose that $u\in L^{t}(\bbC) \cap L^{t^{\prime}}(\bbC)$ { for some $t \in (1,4/3)$}. For any $p>2$ with 
\begin{equation}
\label{pt1}
\frac{1}{2}+\frac{1}{p}<\frac{1}{t} 
\end{equation} 
and
\begin{equation}
\label{pt2}
\frac{1}{p}+\frac{1}{t}>1,
\end{equation}
{ the operator $S_{k,u}$
belongs to $\calE _{p}$} and the determinant%
\[
{ D(k,u) }=\Det\left(  I- S_{k,u}\right)
\]
is a well-defined, bounded continuous function of $k\in\mathbb{C}$ { and $u \in L^t(\bbC) \cap L^{t'}(\bbC)$.}  Moreover,
${ D(k,u)} \rightarrow1$ as $\left\vert k\right\vert \rightarrow\infty$
{ 
and 
\begin{equation}
\label{D.unif}
\sup_{k \in \bbC} |D(k,u) - D(k,u')| \lesssim_{\, p, t} \norm{u-u'}{L^t \cap L^{t'}}
\end{equation}
where the implied constant depends on $p$, $t$, and a bound on $\norm{u}{L^t \cap L^{t'}}$ 
and $\norm{u'}{L^t \cap L^{t'}}$.
}
\end{proposition}

\begin{remark}
\label{rem:pt}
The conditions \eqref{pt1} and \eqref{pt2} dictate that $t \in (1,4/3)$. The figure below shows the region of admissible $p$ and $t$.
\end{remark}

\begin{figure}[H]
\caption{Admissible Values of $t$ and $p$}
\medskip
\begin{tikzpicture}[scale=4]
\draw[thick,->,>=stealth] (0,0) -- (0,1.5);
\draw[thick,->,>=stealth] (0,0) -- (1.5,0);
\draw[dashed] (0,0.5) -- (1,1.5);				
\draw[dashed] (0,1) -- (1,0);						
\draw[dashed] (0.5,0) -- (0.5,1.5);				
\draw[dashed] (0,1) -- (1.5,1);					
\node[left] at (-0.1,1.5) 		{{$1/t$}};
\node[below] at (1.5,-0.1) 	{{$1/p$}};
\node[left] at (0,1) 				{\footnotesize{$(0,1)$}};
\node[right,fill=white] at (0.275,0.75) 	{\footnotesize{$(\frac{1}{4},\frac{3}{4})$}};
\node[right,fill=white] at (0.55,1) 		{\footnotesize{$(\frac{1}{2},1)$}};
\draw[black,thick,fill=yellow] (0.25,0.75) -- (0.5,1) -- (0,1) --(0.25,0.75);
\node[right] at (1.55,1) 			{{$t=1$}};
\node[below] at (0.5,-0.1)	{{$p=2$}};
\node[right]   at (1,1.5)			{{$\frac{1}{t} = \frac{1}{2}+\frac{1}{p}$}};
\node[right]	at (1,0.2)		{{$\frac{1}{p} + \frac{1}{t} = 1$}};
\end{tikzpicture}
\label{fig:tp}
\end{figure}
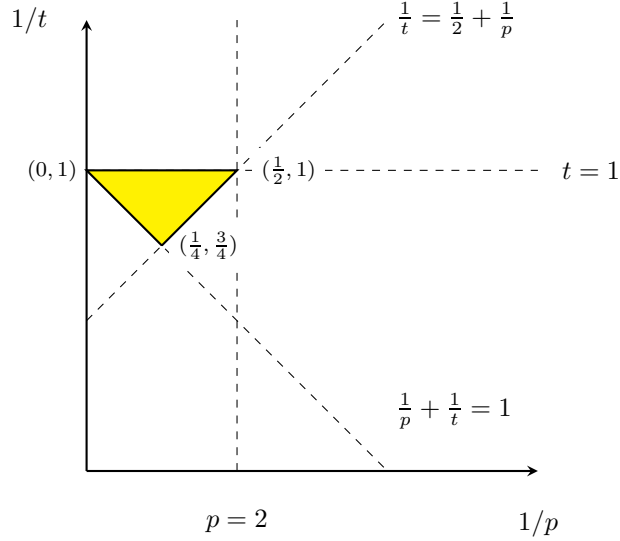

The integral kernel of $S_{k,u}$ is
\[
a(z,w)=\left[  \frac{1}{\pi^{2}}\int\frac{1}{z-\zeta}e_{-k}\left(
\zeta\right)  u(\zeta)\frac{1}{\overline{\zeta}-\overline{w}}~d\zeta\right]
e_{k}(w)\overline{u}(w)~
\]
{ In order to show that $S_{k,u} \in \calE_p$, we need to bound 
$\norm{a}{L^p(L^{p'})}$ and $\norm{a^*}{L^{p'}(L^p)}$.
}
Note that%
\begin{equation}
\left\vert a(z,w)\right\vert \leq\pi^{-2}\int\frac{1}{\left\vert
z-\zeta\right\vert }\frac{1}{\left\vert \zeta-w\right\vert }\left\vert
u(\zeta)\right\vert \left\vert u\left(  w\right)  \right\vert ~dm\left(
\zeta\right)  . \label{a.abs}%
\end{equation}
and%
\begin{equation}
\left\vert a^{\ast}(z,w)\right\vert \leq\pi^{-2}\int\frac{1}{\left\vert
w-\zeta\right\vert }\frac{1}{\left\vert \zeta-z\right\vert }\left\vert
u(\zeta)\right\vert \left\vert u\left(  z\right)  \right\vert ~dm\left(
\zeta\right)  \label{a.star.abs}%
\end{equation}
We will find conditions on $u$ so that $S_{k,u}$ belongs to the
Mikhlin-Itskovich algebra. Before doing so we collect some preliminary
estimates. For a measurable function $g$ on $\mathbb{C}\times\mathbb{C}$,
define%
\begin{align*}
I_{1}(g)(z,w)  &  =\int\frac{g(\zeta,w)}{\left\vert \zeta-z\right\vert
}~d\zeta,\\
I_{2}(g)(z,w)  &  =\int\frac{g(z,\zeta)}{\left\vert \zeta-w\right\vert
}~d\zeta.
\end{align*}
(the \textquotedblleft1\textquotedblright\ and \textquotedblleft%
2\textquotedblright\ refer to integration with respect to the first or second
argument of $g$).

\begin{lemma}
The estimates%
\begin{align}
\left\Vert I_{1}(g)\right\Vert _{L^{\widetilde{p}}\left(  L^{q}\right)  }  
	&	\lesssim_{\, p,q}
		\left\Vert g\right\Vert _{L^{p}\left(  L^{q}\right)}, 
	&& p \in (1,2), \, \, q \in (1,\infty) 
	\label{I1.mix}\\
\left\Vert I_{1}(g)\right\Vert _{L^{\infty}\left(  L^{q}\right)  }  
	&  \lesssim_{\, p,q}
		\left\Vert g\right\Vert _{L^{p}\cap L^{p^{\prime}}\left(  L^{q}\right)},
	&&	p \in (1,2), \, \, q \in (1,\infty),
	\label{I1.infty}\\
\left\Vert I_{2}(g)\right\Vert _{L^{p}\left(  L^{\widetilde{q}}\right)  }  
	&	\lesssim_{\, p,q}
		\left\Vert g\right\Vert _{L^{p}\left(  L^{q}\right)},
	&&	p \in (1,\infty), \, \, q \in (1,2),
	\label{I2.mix}\\
\left\Vert I_{2}(g)\right\Vert _{L^{p}\left(  L^{\infty}\right)  }  
	& \lesssim_{\, p,q}  
		\left\Vert g\right\Vert _{L^{p}\left(  L^{q}\cap L^{q^{\prime}}\right)},
	&&	p, q \in (1,\infty) 
	\label{I2.infty}%
\end{align}
hold.
\end{lemma}

\begin{proof}
To prove (\ref{I1.mix}), we use Minkowski's integral inequality and 
Remark \ref{rem:I1} to estimate%
\begin{align*}
\left( 
	 \int
	 	\left(  
	 		\int
	 			\left\vert \int\frac{g(\zeta,w)}{|\zeta-z|}~d\zeta\right\vert ^{q}
	 		~dw
	 	\right)^{\widetilde{p}/q}~dz
\right)^{1/\widetilde{p}}  
&  \leq
\left(  
	\int
		\left(  
			\int
				\frac{\left\Vert g\left(  \zeta,~\cdot\right)  \right\Vert_{q}}{\left\vert\zeta-z\right\vert }			
			~d\zeta
		\right)^{\widetilde{p}}dz
\right) ^{1/\widetilde{p}}\\
&  \lesssim_{\, p}
\left(  
	\int
		\left\Vert g(\zeta,~\cdot~)\right\Vert _{q}^{p}
	~d\zeta
\right)^{1/p}\\
&  =C_{p}
	\left\Vert g\right\Vert _{L^{p}\left(  L^{q}\right)  }.
\end{align*}

To prove (\ref{I1.infty}), we use the Hardy-Littlewood Sobolev inequality and Minkowski's integral inequality to estimate%
\begin{align*}
\left(  \int\left\vert \int\frac{g(\zeta,w)}{|\zeta-z|}~d\zeta\right\vert
^{q}~dw\right)^{1/q}   &  \lesssim_{\, q}\int\frac{1}{\left\vert \zeta-z\right\vert
}\left\Vert g(\zeta,~\cdot~)\right\Vert _{q}~d\zeta\\
&  \lesssim_{\, q}\left\Vert g\right\Vert _{L^{p}\cap L^{p^{\prime}}\left(
L^{q}\right)  }.
\end{align*}

To prove (\ref{I2.mix}), we use the Hardy-Littlewood-Sobolev inequality to
estimate%
\begin{align*}
\left\Vert I_{2}(g)\right\Vert _{L^{p}\left(  L^{\widetilde{q}}\right)  }  &
\lesssim_{\, q}\left(  \int\left\Vert g(z,~\cdot~)\right\Vert _{q}^{p}~dz\right)
^{1/p}\\
&  =\left\Vert g\right\Vert _{L^{p}\left(  L^{q}\right)  }.
\end{align*}

To prove (\ref{I2.infty}), we use Remark \ref{rem:I1} to estimate%
\begin{align*}
\left\Vert I_{2}(g)\right\Vert _{L^{p}\left(  L^{\infty}\right)  }  &  \lesssim_{\, p}
\left(  \int\left\Vert g(z,~\cdot~)\right\Vert _{L^{q}\cap L^{q^{\prime}%
}}^{p}~dz\right)  ^{1/p}\\
&  =\left\Vert g\right\Vert _{L^{p}\left(  L^{q}\cap L^{q^{\prime}%
}\right)  }.
\end{align*}

\end{proof}

\begin{lemma}
Suppose that $u\in L^{t}\cap L^{t^{\prime}}$ where $1\leq t<2$. Then, for any
$p>2$ { satisfying \eqref{pt1}--\eqref{pt2}},
we have%
\[
a\in L^{p}\left(  L^{p^{\prime}}\right)  ,~a^{\ast}\in L^{p^{\prime}}\left(
L^{p}\right)  \text{.}%
\]
with%
\begin{equation}
\max\left(  \left\Vert a\right\Vert _{L^{p}\left(  L^{p^{\prime}}\right)
},\left\Vert a^{\ast}\right\Vert _{L^{p^{\prime}}\left(  L^{p}\right)
}\right)  \lesssim_{\, p}\left\Vert u\right\Vert _{L^{t}\cap L^{t^{\prime}}}^{2}
\label{a.pnorm}%
\end{equation}
for a constant $C$ independent of $k$. Moreover,
\begin{equation}
\lim_{\left\vert k\right\vert \rightarrow\infty}\left\Vert a\right\Vert
_{L^{p}\left(  L^{p^{\prime}}\right)  }=\lim_{\left\vert k\right\vert
\rightarrow\infty}\left\Vert a^{\ast}\right\Vert _{L^{p^{\prime}}\left(
L^{p}\right)  }=0. \label{a.klim}%
\end{equation}

\end{lemma}

\begin{proof}
To estimate $\left\Vert a\right\Vert _{L^{p}\left(  L^{p^{\prime}}\right)  }$,
we use (\ref{holder}). Let $g\in L^{p^{\prime}}\left(  L^{p}\right)  $ with
$\left\Vert g\right\Vert _{L^{p^{\prime}}\left(  L^{p}\right)  }\leq1$. A
short computation using (\ref{a.abs}) shows that, up to absolute numerical
constants,%
\begin{align}
\left\vert \int a(z,w)g(z,w)~dz~dw\right\vert  &  \leq\int\left\vert
u(\zeta)\right\vert \int\frac{\left\vert u(w)\right\vert }{\left\vert
\zeta-w\right\vert }\int\frac{\left\vert g(z,w)\right\vert }{\left\vert
\zeta-z\right\vert }~dz~dw~d\zeta\label{a.pre}\\
&  \leq\int\left\vert u(\zeta)\right\vert ~\left[  I_{2}\left(  u~\cdot
~I_{1}(g)\right)  \right]  \left(  \zeta,\zeta\right)  ~d\zeta\nonumber
\end{align}
where
\[
\left(  u\cdot I_{1}(g)\right)  (z,w)=u(w)\cdot I_{1}(g)(z,w).
\]
By (\ref{I1.mix}) we have $I_{1}(g)\in L^{s}\left(  L^{p}\right)  $ with
$\left\Vert I_{1}(g)\right\Vert _{L^{s}\left(  L^{p}\right)  }\leq C\left\Vert
g\right\Vert _{L^{p^{\prime}}\left(  L^{p}\right)  }$ where
\[
\frac{1}{s}=\frac{1}{2}-\frac{1}{p}.
\]
By H\"{o}lder's inequality, using the fact that $u\in L^{\infty}\left(
L^{t^{\prime}}\cap L^{t}\right)  $ (viewed as a function of two variables
depending only on $w$) we then have $\left\Vert u(~\cdot~)I_{1}(g)\right\Vert
_{L^{s}\left(  L^{r}\right)  }\leq C\left\Vert u\right\Vert _{L^{t^{\prime}%
}\cap L^{t}}\left\Vert g\right\Vert _{L^{p^{\prime}}\left(  L^{p}\right)  }$
provided $\dfrac{1}{r}$ belongs to the interval $J_{1}=\left(  \dfrac{1}%
{p}+\dfrac{1}{t^{\prime}},\dfrac{1}{p}+\dfrac{1}{t}\right)  $. We claim that
there is an $r\in\left(  1,2\right)  $ with $\left(  \dfrac{1}{r^{\prime}%
},\dfrac{1}{r}\right)  \subset  J_{1}$. Such an $r$ exists provided%
\[
\dfrac{1}{p}+\dfrac{1}{t^{\prime}}<\dfrac{1}{2},~\dfrac{1}{p}+\dfrac{1}{t}>\dfrac
{1}{2}.
\]
The second inequality is trivial since $t<2$ and the first is equivalent to
(\ref{pt1}). Choosing such an $r$ we now have $u\left(  ~\cdot\right)
I_{1}(g)\in L^{s}\left(  L^{r}\cap L^{r^{\prime}}\right)  $. Now we use
(\ref{I2.infty}) to conclude that $I_{2}\left(  u\cdot I_{1}(g)\right)  \in
L^{s}\left(  L^{\infty}\right)  $ so that, by (\ref{diag}), $\left[
I_{2}\left(  u~\cdot~I_{1}(g)\right)  \right]  \left(  \zeta,\zeta\right)  \in
L^{s}$ with
\[
\left\Vert I_{2}(u\cdot I_{1}(g)(~\cdot~,~\cdot~)\right\Vert _{L^{s}}\leq
C\left\Vert u\right\Vert _{L^{t^{\prime}}\cap L^{t}}\left\Vert g\right\Vert
_{L^{p^{\prime}}\left(  L^{p}\right)  }.
\]
Hence, we can bound the right-hand side of (\ref{a.pre}) by $\left\Vert
u\right\Vert _{s^{\prime}}\left\Vert I_{2}\left(  u\cdot I_{1}(g)\right)
\right\Vert _{L^{s}\left(  L^{\infty}\right)  }$ provided $\dfrac{1}{s^{\prime
}}\in\left(  \dfrac{1}{t^{\prime}},\dfrac{1}{t}\right)  $. As $\dfrac
{1}{s^{\prime}}=\dfrac{1}{2}+\dfrac{1}{p}$ we need the two inequalities%
\[
\dfrac{1}{2}+\frac{1}{p}<\frac{1}{t},~\frac{1}{2}+\frac{1}{p}>1-\frac{1}{t}%
\]
to hold. The first is (\ref{pt1}) and the second is equivalent to $\dfrac
{1}{t}>\dfrac{1}{2}-\dfrac{1}{p}$ which is trivial since $t<2$. Hence%
\[
\left\Vert a\right\Vert _{L^{p}\left(  L^{p^{\prime}}\right)  }\leq
C\left\Vert u\right\Vert _{L^{t^{\prime}}\cap L^{t}}^{2}%
\]

Next, to estimate $\left\Vert a^{\ast}\right\Vert _{L^{p^{\prime}}\left(
L^{p}\right)  }$, we choose $g\in L^{p}\left(  L^{p^{\prime}}\right)  $ with
$\left\Vert g\right\Vert _{L^{p}\left(  L^{p^{\prime}}\right)  }\leq1$. We
then use (\ref{holder}) and (\ref{a.star.abs}) to bound (again up to
numerical constants)%
\begin{align}
\left\vert \int a^{\ast}g~dz~dw\right\vert  &  \leq\int\left\vert
u(z)\right\vert \int\frac{\left\vert u(\zeta)\right\vert }{\left\vert
\zeta-z\right\vert }\int\frac{\left\vert g(z,w)\right\vert }{\left\vert
\zeta-w\right\vert }~dw~d\zeta~dz\label{astar.pre}\\
&  =\int\left\vert u(z)\right\vert ~I_{2}\left(  u\cdot I_{2}(g)\right)
\left(  z,z\right)  ~dz\nonumber
\end{align}
where%
\[
\left(  u\cdot I_{2}\left(  g\right)  \right)  \left(  z,w\right)
=u(w)I_{2}(g)\left(  z,w\right)  .
\]
First, by (\ref{I2.mix}), we have $I_{2}(g)\in L^{p}\left(  L^{s}\right)  $
with $\left\Vert I_{2}(g)\right\Vert _{L^{p}\left(  L^{s}\right)  }\leq
C\left\Vert g\right\Vert _{L^{p}\left(  L^{p^{\prime}}\right)  }$, where
$\dfrac{1}{s}=\dfrac{1}{2}-\dfrac{1}{p}$. Since $u$ (viewed as a function of two
variables depending only on the second variable)\ belongs to $L^{\infty
}\left(  L^{t}\cap L^{t^{\prime}}\right)  $, it follows that $u\cdot I_{2}(g)$
belongs to $L^{p}\left(  L^{r}\right)  $ for any $r$ with $\dfrac{1}{r}$
belonging to the interval $J_2=\left(  \dfrac{1}{s}+\dfrac{1}{t^{\prime}%
},\dfrac{1}{s}+\dfrac{1}{t}\right)  $ and $\left\Vert u\cdot I_{2}(g)\right\Vert
_{L^{p}\left(  L^{r}\right)  }\lesssim_{\, p}\left\Vert u\right\Vert _{L^{t}\cap
L^{t^{\prime}}}\,\left\Vert g\right\Vert _{L^{p^{\prime}}\left(  L^{p}\right)
}$. We claim that there is an $r\in(1,2)$ with $\left(  \dfrac{1}{r^{\prime}%
},\dfrac{1}{r}\right)  \in J_2$. This is the case provided the two
inequalities%
\[
\frac{1}{s}+\frac{1}{t^{\prime}}<\frac{1}{2},~~\frac{1}{s}+\frac{1}{t}%
>\frac{1}{2}%
\]
hold. The first is equivalent to (\ref{pt2}). The second inequality is
trivial since $t<2$. We can now use (\ref{I2.infty}) to estimate%
\begin{align*}
\left\Vert I_{2}\left(  u~\cdot~I_{2}(g)\right)  \right\Vert _{L^{p}\left(
L^{\infty}\right)  }  &  \leq C\left\Vert u~\cdot~I_{2}(g)\right\Vert
_{L^{p}\left(  L^{r}\cap L^{r^{\prime}}\right)  }\\
&  \leq C\left\Vert u\right\Vert _{L^{t}\cap L^{t^{\prime}}}\left\Vert
g\right\Vert _{L^{p}\left(  L^{p^{\prime}}\right)  }.
\end{align*}
Finally, using (\ref{diag}) and H\"{o}lder's inequality, we can bound the
right-hand side of (\ref{astar.pre}) by%
\[
C\left\Vert u\right\Vert _{L^{p^{\prime}}}\left\Vert u\right\Vert _{L^{t}\cap
L^{t^{\prime}}}\left\Vert g\right\Vert _{L^{p}\left(  L^{p^{\prime}}\right)  }%
\]
which is in turn bounded by $C\left\Vert u\right\Vert _{L^{t}\cap
L^{t^{\prime}}}^{2}$ provided $\dfrac{1}{p^{\prime}}\in\left(  \dfrac
{1}{t^{\prime}},\dfrac{1}{t}\right)  $. This is true provided $\dfrac
{1}{p^{\prime}}>\dfrac{1}{t^{\prime}}$ and $\dfrac{1}{p^{\prime}}<\dfrac{1}{t}$.
The first of these inequalities is trivial since $p>t$ and the second is
equivalent to (\ref{pt2}).

To prove (\ref{a.klim}), it suffices to show that the limits are zero in
case $u\in\mathcal{C}_{0}^{\infty}(\mathbb{C})$. Emphasizing the dependence of
$a$ on $k$, write%
\[
a\left(  z,w,k\right)  =\left[  \dfrac{1}{\pi^{2}}\int\frac{1}{z-\zeta}%
e_{-k}\left(  \zeta\right)  u(\zeta)\frac{1}{\overline{\zeta}-\overline{w}%
}~d\zeta\right]  e_{k}(w)\overline{u}(w).
\]
For each fixed $z,w$ it follows from the Riemann-Lebesgue lemma that%
\[
\lim_{\left\vert k\right\vert \rightarrow\infty}a(z,w,k)=0
\]
for almost every $\left(  z,w\right)  $. Since $\left\vert a(z,w,k)\right\vert
$ is dominated by a fixed $L^{p^{\prime}}\left(  L^{p}\right)  $ function, it
follows that $\left\Vert a(~\cdot~,~\cdot~,k)\right\Vert _{L^{p}\left(
L^{p^{\prime}}\right)  }\rightarrow0$ as $\left\vert k\right\vert
\rightarrow\infty$. A similar argument shows that $\left\Vert a^{\ast}\left(
~\cdot~,~\cdot~,k\right)  \right\Vert _{L^{p^{\prime}}\left(  L^{p}\right)
}\rightarrow0$ as $\left\vert k\right\vert \rightarrow\infty$.
\end{proof}


\begin{proof}[Proof of Proposition \ref{prop:brown}]
{The continuity follows from the fact that the maps%
\begin{align*}
(k,u)  &  \rightarrow a(z,w,k)\\
(k,u)  &  \rightarrow a^{\ast}(z,w,k)
\end{align*}
respectively from 
$(L^{t}(\bbC) \cap L^{t'}(\bbC)) \times \mathbb{C}$ to $L^{p}\left(  L^{p^{\prime}}\right)  $ 
and
to $L^{p^{\prime}}\left(  L^{p}\right)  $ are continuous. The fact that
$D(k)\rightarrow1$ follows from the fact that $\left\Vert S_{k,u}\right\Vert
_{\mathcal{E}_{p}}\rightarrow0$ as $\left\vert k\right\vert \rightarrow\infty
$, as follows from (\ref{a.klim}). The estimate \eqref{D.unif} follows from the 
bilinearity of $u \mapsto S_{k,u}$ and the fact that estimates on $\norm{a(\dotarg,\dotarg,k)}{L^p(L^{p'})}$ and $\norm{a^*(\dotarg,\dotarg,k)}{L^{p'}(L^p)}$ are independent of $k$.
}
\end{proof}

\end{document}